\documentclass[final,leqno,onefignum,onetabnum]{siamltex1213}

\usepackage{amssymb}
\usepackage{amsmath}
\usepackage{amsfonts}

\newtheorem{remark}{Remark}[section]
\newtheorem{example}{Example}[section]

\newcommand{\limfunc}[1]{{\rm #1\hspace{0.15em}}}

\title{On the two mutually independent factors that determine the convergence of least-squares projection method
\thanks{
The manuscript of this paper could be found at arXiv:$1406.0578$.
}}

\author{
Shukai Du
\thanks{Corresponding author. School of Mathematics and Statistics, Wuhan University, Wuhan, People's Republic of China.
(\email{shukai.du@whu.edu.cn}).}
\and Nailin Du
\thanks{School of Mathematics and Statistics, Wuhan University, Wuhan, People's Republic of China.
(\email{nldu.math@whu.edu.cn}).}}

\begin{document}
\maketitle
\slugger{simax}{xxxx}{xx}{x}{x--x}

\begin{abstract}
The paper investigates the least-squares projection method for bounded
linear operators, which provides a natural regularization scheme by
projection for many ill-posed problems. Yet, without additional assumptions,
the convergence of this approximation scheme cannot be guaranteed. We reveal
that the convergence of least-squares projection method is determined by two
mutually independent factors -- the \textquotedblleft kernel
approximability\textquotedblright\ and the \textquotedblleft offset
angle\textquotedblright . The kernel approximability is a necessary
condition of convergence described with kernel $\mathcal{N}\left( T\right) $
and its subspaces $\mathcal{N}\left( T\right) \cap X_{n}$, and we give
several equivalent characterizations for it (Theorem \ref{the1.0}). The
offset angle of $X_{n}$ is defined as the largest canonical angle between
space $T^{\ast }T(X_{n})$ and $T^{\dagger }T(X_{n})$ (which are subspaces of
$\mathcal{N}\left( T\right) ^{\bot }$), and it geometrically reflects the
rate of convergence (Theorem \ref{the1.1}). The paper also presents new
observations\ for the unconvergence examples of Seidman \cite[Example 3.1]%
{seidman1980nonconvergence} and Du \cite[Example 2.10]{du2008finite} under
the notions of kernel approximability and offset angle.
\end{abstract}



\pagestyle{myheadings}
\thispagestyle{plain}
\markboth{Convergence of Least-squares Projection Method}{Shukai Du and Nailin Du}

\section{Introduction}
In this paper we investigate the least-squares projection method for bounded
linear operators. As is generally known, such investigations are the bases
for numerically solving operator equations of the first and the second kind
(see Du \cite{nailin2005basic} and \cite{du2008finite}, Groetsch-Neubauer
\cite{Groetsch-Neubauer1988}, Groetsch \cite{Groetsch1982}, Luecke-Hickey
\cite{Luecke-Hickey1985}, Seidman \cite{seidman1980nonconvergence},
Spies-Temperini \cite{spies2006arbitrary}, Engl-Hank-Neubauer \cite{Engl2000}%
, Kress \cite{Kress1999}, and the references cited therein).

Let $X$ and $Y$ be Hilbert spaces, and $T\in \mathcal{B}\left( X,Y\right) $,
i.e., $T:X\rightarrow Y$ be a bounded linear operator. Let $T^{\dag }$ and $%
T^{\ast }$ denote its Moore-Penrose inverse and adjoint operator, $\mathcal{D%
}\left( T\right) $, $\mathcal{N}\left( T\right) $, $\mathcal{R}\left(
T\right) $, and $\mathcal{G}\left( T\right) $ denote its domain, kernel,
range, and graph, respectively. If $X_{0}$ and $Y_{0}$ are closed subspaces
of $X$ and $Y$, $P_{X_{0}}$ and $P_{Y_{0}}$ will stand for the orthogonal
projections from $X$ onto $X_{0}$ and from $Y$ onto $Y_{0}$, respectively.
Let $X\times Y$ be the Hilbert space with the inner product defined by%
\begin{equation*}
\left\langle \left( x_{1},y_{1}\right) ,\left( x_{2},y_{2}\right)
\right\rangle :=\left\langle x_{1},x_{2}\right\rangle +\left\langle
y_{1},y_{2}\right\rangle \text{\quad }\forall \left( x_{1},y_{1}\right)
,\left( x_{2},y_{2}\right) \in X\times Y.
\end{equation*}

Let $\left\{ X_{n}\right\} $ be a sequence of finite-dimensional subspaces
of $X$ with $\dim X=\infty $ such that%
\begin{equation*}
\underset{n\rightarrow \infty }{\mathrm{s\hspace{0.15em}}\!\text{%
-\negthinspace }\lim }P_{X_{n}}=I_{X},
\end{equation*}%
and set%
\begin{equation*}
T_{n}:=TP_{X_{n}}\quad \forall n\in
\mathbb{N}
\text{.}
\end{equation*}%
We say $\{(X_{n},T_{n})\}_{n\in
\mathbb{N}
}$ is a LPA (least-squares projection approximation) for $T$. All of our
discussions in this paper will be based on this setting. Note that, if $%
\left\{ X_{n}\right\} $ is an increasing sequence%
\begin{equation*}
X_{0}\subseteq X_{1}\subseteq X_{2}\subseteq \cdots \text{\quad with}\quad
\overline{\overset{\infty }{\underset{n=0}{\cup }}X_{n}}=X\,,
\end{equation*}%
then $\{(X_{n},T_{n})\}$ is a natural LPA for $T$.

We are interested in approximating $T^{\dag }$ by $T_{n}^{\dag }$ when $\dim
\mathcal{R}\left( T\right) =\infty $, which is the least-squares projection
method for $T$. So the issue of convergence of LPA $\{(X_{n},T_{n})\}$
\begin{equation}
\quad \underset{n\rightarrow \infty }{\mathrm{w\hspace{0.15em}}\text{-}\lim }%
T_{n}^{\,\dag }=T^{\dag }\text{ on }\mathcal{D}\left( T^{\dag }\right) \quad
\text{(\emph{Weak Convergence})}  \label{eqn1.1}
\end{equation}%
and%
\begin{equation}
\quad \underset{n\rightarrow \infty }{\mathrm{s\hspace{0.15em}}\text{-}\lim }%
T_{n}^{\,\dag }=T^{\dag }\text{ on }\mathcal{D}\left( T^{\dag }\right) \quad
\text{(\emph{Strong Convergence})}  \label{eqn1.2}
\end{equation}%
naturally arise. Note that without additional assumptions it cannot be
guaranteed that (\ref{eqn1.1}) or (\ref{eqn1.2}) holds, as Seidman's example
\cite[Example 3.1]{seidman1980nonconvergence} and Du's example \cite[Example
2.10]{du2008finite} show. By Groetsch \cite[Proposition 0]{Groetsch1982} and
Du \cite[Theorem 2.6 with Table 4.1]{du2008finite} there exists the
following convergence result:%
\begin{equation*}
\text{(\ref{eqn1.1})}\Longleftrightarrow \text{(\ref{eqn1.2})}%
\Longleftrightarrow \text{(\ref{eqn1.3})},
\end{equation*}%
where
\begin{equation}
\sup_{n}\left\Vert T_{n}^{\dag }T\right\Vert <+\infty .  \label{eqn1.3}
\end{equation}%
However, as condition (\ref{eqn1.3}) lacks geometric intuition, it is still
difficult for us to choose a suitable LPA such that (\ref{eqn1.3}) holds. By
Du \cite[Theorem 2.2]{du2008finite}, (\ref{eqn1.3}) implies that%
\begin{equation*}
\underset{n\rightarrow \infty }{\mathrm{s\hspace{0.15em}}\text{-}\lim }P_{%
\mathcal{N}\left( T_{n}\right) }=P_{\mathcal{N}\left( T\right) },
\end{equation*}%
which is equivalent to
\begin{equation}
\mathcal{N}\left( T\right) =\left\{ x\in X\,:\,\lim_{n\rightarrow \infty }%
\mathrm{dist\hspace{0.15em}}\left( x,\mathcal{N}\left( T\right) \cap
X_{n}\right) =0\right\} .  \label{eqn1.4}
\end{equation}%
Hence, (\ref{eqn1.4}) is necessary in choosing a suitable LPA for $T$,
namely, (\ref{eqn1.4}) is a necessary condition of (\ref{eqn1.3}). Here, we
remark that (\ref{eqn1.4}) does not naturally hold (see \cite[Example 2.10]%
{du2008finite}), but it is more likely to be satisfied compared to (\ref%
{eqn1.3}) (for instance, when $T$ is injection, (\ref{eqn1.4}) always
holds). However, as (\ref{eqn1.4}) is not a sufficient condition of (\ref%
{eqn1.3}) (see \cite[Example 3.1]{seidman1980nonconvergence}), we still need
some complementary conditions with which (\ref{eqn1.4}) could lead to (\ref%
{eqn1.3}).

In this paper, we aim to find the complementary condition which together
with (\ref{eqn1.4}) constitute a necessary and sufficient condition of the
convergence of least-squares projection method. We hope that the
complementary condition has enough geometrical meanings, so as to give us
new insight about least-squares projection method. We obtain the following
main results:

\begin{itemize}
\item Several equivalent conditions of (\ref{eqn1.4}) are given. Notice that%
\begin{equation*}
\mathcal{N}\left( T_{n}\right) =\mathcal{N}\left( T\right) \cap
X_{n}+X_{n}^{\bot },
\end{equation*}%
therefore, $\mathcal{N}\left( T\right) \cap X_{n}$ will be referred to as
the core of $\mathcal{N}\left( T_{n}\right) $, and the condition (\ref%
{eqn1.4}) will be called as \emph{kernel approximability}$\mathcal{\ }$of
LPA $\left\{ \left( X_{n},T_{n}\right) \right\} $.

\item A new concept called \textquotedblleft \emph{offset
angle\textquotedblright\ }of\emph{\ }$X_{n}$\ is introduced to describe the
complementary condition, which is defined as the largest canonical angle
between space $T^{\ast }T(X_{n})$ and $T^{\dagger }T(X_{n})$. In the case of
$\dim \mathcal{N}\left( T\right) <\infty $, we show that LPA $\left\{ \left(
X_{n},T_{n}\right) \right\} $ is convergent if and only if it has the kernel
approximability and the supreme of all offset angles is less than
perpendicular angle. Moreover, if LPA $\left\{ \left( X_{n},T_{n}\right)
\right\} $ is with kernel approximability, the rate of convergence of LPA $%
\left\{ \left( X_{n},T_{n}\right) \right\} $ is geometrically reflected by
the offset angles.

\item The classical unconvergence example of Seidman is restudied under the
concept of offset angle, and we show that the reason for unconvergence is
actually caused by offset angle of\emph{\ }$X_{n}$ tending towards
perpendicular angle. On the other hand, the unconvergence example of Du is
also restudied, which is with constant zero offset angle, and we show that
the reason for unconvergence is caused by (\ref{eqn1.4}) (kernel
approximability) becoming invalid.
\end{itemize}

In order to expound our main results more precisely, the following notations
are needed: If $\{S_{n}\}$ is a sequence of nonempty subsets of a Banach
space, define%
\begin{equation*}
\begin{array}{c}
\underset{n\rightarrow \infty }{\mathrm{s\hspace{0.15em}}\text{-}\lim }%
S_{n}:=\left\{ x:\text{there is a sequence }\left\{ x_{n}\right\} \text{
such that }S_{n}\ni x_{n}\rightarrow x\right\} ,\medskip  \\
\underset{n\rightarrow \infty }{\mathrm{w\hspace{0.15em}}\text{-}\widetilde{%
\lim }}S_{n}:=\left\{ x:\text{there is a sequence }\left\{ x_{n}\right\}
\text{ such that }\overset{\infty }{\underset{k=n}{\cup }}S_{k}\ni
x_{n}\rightharpoonup x\right\} .%
\end{array}%
\end{equation*}%
If $M$ and $N$ are both closed subspaces of a Hilbert space $H$, define%
\begin{equation*}
\limfunc{gap}\left( M,N\right) :=\max \left\{ \delta \left( M,N\right)
,\delta \left( N,M\right) \right\} ,
\end{equation*}%
where%
\begin{equation*}
\delta \left( M,N\right) :=\left\{
\begin{array}{r}
\medskip \sup \left\{ \mathrm{dist\hspace{0.15em}}\left( x,N\right) :x\in
M,\left\Vert x\right\Vert =1\right\} ,\quad \text{if }M\neq \left\{
0\right\} , \\
0\,,\quad \text{if }M=\left\{ 0\right\} .%
\end{array}%
\right.
\end{equation*}%
$\limfunc{gap}\left( M,N\right) $ is called the \emph{gap} between $M$ and $N
$ (see \cite{Kato1980}). When%
\begin{equation*}
m:=\dim \left( M\right) \leq \dim \left( N\right) <\infty ,
\end{equation*}%
the canonical angles (or principal angles) between $M$ and $N$ can be
defined, which are a sequence of $m$ angles $0\leq \vartheta _{1}\leq
\vartheta _{2}\leq ...\leq \vartheta _{m}\leq \frac{\pi }{2}$. By \cite%
{golub2012matrix}, the canonical angles are defined recursively by%
\begin{equation*}
\cos \vartheta _{k}=\max_{u\epsilon M}\max_{v\epsilon N}\left\langle
u,v\right\rangle =\left\langle u_{k},v_{k}\right\rangle
\end{equation*}%
subject to%
\begin{equation*}
\left\{
\begin{array}{l}
\left\Vert u\right\Vert =\left\Vert v\right\Vert =1, \\
\left\langle u,u_{i}\right\rangle =0\qquad i=1\colon k-1, \\
\left\langle v,v_{i}\right\rangle =0\qquad i=1\colon k-1,%
\end{array}%
\right.
\end{equation*}%
and if $\dim \left( M\right) =\dim \left( N\right) <\infty $, the largest
canonical angle $\vartheta _{m}$ satisfies
\begin{equation}
\sin \vartheta _{m}=\limfunc{gap}\left( M,N\right) .  \label{eqn1.4_1}
\end{equation}

With the above notions, the main theorems of the paper are stated as below:

\begin{theorem}
\label{the1.0} Let $T\in \mathcal{B}\left( X,Y\right) $ have LPA $%
\{(X_{n},T_{n})\}_{n\in
\mathbb{N}
}$. Then the following conditions are equivalent:

\textrm{(a)} \emph{Kernel Approximability}: (\ref{eqn1.4}) is valid, namely,%
\begin{equation*}
\underset{n\rightarrow \infty }{\mathrm{s\hspace{0.15em}}\text{-}\lim }%
\mathcal{N}\left( T_{n}\right) =\underset{n\rightarrow \infty }{\mathrm{w%
\hspace{0.15em}}\text{-}\widetilde{\lim }}\mathcal{N}\left( T_{n}\right) =%
\mathcal{N}\left( T\right) .
\end{equation*}

\textrm{(b)} \emph{Inverse-graph Approximability}:%
\begin{equation*}
\underset{n\rightarrow \infty }{\mathrm{s\hspace{0.15em}}\text{-}\lim }%
\mathcal{G}\left( T_{n}^{\,\dag }\right) =\underset{n\rightarrow \infty }{%
\mathrm{w\hspace{0.15em}}\text{-}\widetilde{\lim }}\mathcal{G}\left(
T_{n}^{\,\dag }\right) =\mathcal{G}\left( T^{\dag }\right) .
\end{equation*}

\textrm{(c)} \emph{Bounded-weak Convergence}: If a sequence $\left\{
y_{n}\right\} \subseteq Y$ satisfies%
\begin{equation*}
\sup_{n}\left\Vert T_{n}^{\dag }y_{n}\right\Vert <+\infty \text{\quad
and\quad }\underset{n\rightarrow \infty }{\mathrm{w\hspace{0.15em}}\text{-}%
\lim }y_{n}=y,
\end{equation*}%
then%
\begin{equation*}
y\in \mathcal{D}\left( T^{\dag }\right) \quad \text{and\quad }\underset{%
n\rightarrow \infty }{\mathrm{w\hspace{0.15em}}\text{-}\lim }T_{n}^{\dag
}y=T^{\dagger }y.
\end{equation*}
\end{theorem}

\begin{theorem}
\label{the1.1} Let $T\in \mathcal{B}\left( X,Y\right) $ with $\dim \mathcal{N%
}\left( T\right) <\infty $ have LPA $\{(X_{n},T_{n})\}_{n\in
\mathbb{N}
}$. Set%
\begin{equation}
\theta _{n}:=\arcsin \limfunc{gap}\left( T^{\dag }T\left( X_{n}\right)
,T^{\ast }T\left( X_{n}\right) \right) \quad \forall n\in
\mathbb{N}
,  \label{eqn1.5}
\end{equation}%
which is called the $\emph{offset}$ $\emph{angle}$ of $X_{n}$ respect to $T$%
. Then the following propositions hold:

\textrm{(a)} $\left\{ \theta _{n}\right\} $ is a sequence in the interval $%
\left[ 0,\frac{\pi }{2}\right) $, and there holds%
\begin{equation*}
\text{(\ref{eqn1.3})}\Longleftrightarrow \left\{
\begin{array}{l}
\text{(\ref{eqn1.4}),} \\
\sup_{n}\theta _{n}<\frac{\pi }{2}\text{.}%
\end{array}%
\right.
\end{equation*}

\textrm{(b)} If (\ref{eqn1.4}) is valid, then there is a $n_{\ast }\in
\mathbb{N}
$ such that, for $n\geq n_{\ast }$,
\begin{equation*}
\left\Vert T_{n}^{\dag }y-T^{\dag }y\right\Vert \leq \sqrt{1+\tan ^{2}\theta
_{n}}\mathrm{dist\hspace{0.15em}}\left( T^{\dag }y,X_{n}\right) \quad
\forall y\in \mathcal{D}\left( T^{\dag }\right) ,
\end{equation*}%
and%
\begin{equation*}
\theta _{n}=0\Longleftrightarrow T_{n}^{\dag }=P_{X_{n}}T^{\dag }\text{ on }%
\mathcal{D}\left( T^{\dag }\right) \Longleftrightarrow \mathcal{N}\left(
T\right) +T^{\ast }T\left( X_{n}\right) \subseteq X_{n}.
\end{equation*}
\end{theorem}

\begin{remark}
As $\dim (T^{\dag }T\left( X_{n}\right) )=\dim (T^{\ast }T\left(
X_{n}\right) )<\infty $, the offset angle $\theta _{n}$ of $X_{n}$ respect
to $T$ is the largest canonical angle between $T^{\dag }T\left( X_{n}\right)
$ and $T^{\ast }T\left( X_{n}\right) $ by (\ref{eqn1.4_1}).
\end{remark}

The proofs of Theorems \ref{the1.0} and \ref{the1.1} are in Section 2 and
Section 3 respectively. In section 4, the two unconvergence examples of
Seidman and Du will be restudied to further explain the relations among the
three concepts of\emph{\ convergence, kernel approximability }and\emph{\ }$%
\emph{offset}$ $\emph{angle}$. Our conclusions will be collected in Section
5.

\section{Proof of Theorem \protect\ref{the1.0}}

To prove Theorem \ref{the1.0}, we need some lemmas.

\begin{lemma}
\label{lem2.0_1} Let $H$ be a Hilbert space and $\left\{ H_{n}\right\} $ a
sequence of closed subspaces of $H$.

$\mathrm{(a)}$ There holds%
\begin{equation*}
\left\{ P_{H_{n}}\right\} \text{ is strongly convergent\ }%
\Longleftrightarrow \,\underset{n\rightarrow \infty }{\mathrm{s\hspace{0.15em%
}}\text{-}\lim }H_{n}=\,\underset{n\rightarrow \infty }{\mathrm{w\hspace{%
0.15em}}\text{-}\widetilde{\lim }}H_{n}\text{,}
\end{equation*}%
and when $\left\{ P_{H_{n}}\right\} $ is strongly convergent, we have
\begin{equation*}
\underset{n\rightarrow \infty }{\mathrm{s\hspace{0.15em}}\text{-}\lim }%
P_{H_{n}}=P_{M}\,,\quad \text{where }M:=\,\underset{n\rightarrow \infty }{%
\mathrm{s\hspace{0.15em}}\text{-}\lim }H_{n}\,.
\end{equation*}

$\mathrm{(b)}$ If $N$ is a closed subspace of $H$, then%
\begin{equation*}
\underset{n\rightarrow \infty }{\mathrm{s\hspace{0.15em}}\text{-}\lim }%
P_{H_{n}}=P_{N}\ \Longleftrightarrow \ \,\underset{n\rightarrow \infty }{%
\mathrm{w\hspace{0.15em}}\text{-}\lim }P_{H_{n}}=P_{N}\medskip .
\end{equation*}
\end{lemma}

\begin{proof}
See \cite[Lemma 2.13]{nailin2005basic} and \cite[Lemma 2.1]{du2008finite}.
\end{proof}

\begin{lemma}
\label{lem2.0} Let $T\in \mathcal{B}\left( X,Y\right) $ have LPA $%
\{(X_{n},T_{n})\}_{n\in
\mathbb{N}
}$.

$\mathrm{(a)}$ \ There hold%
\begin{equation}
\left\{
\begin{array}{l}
\mathcal{R}\left( T_{n}\right) =T\left( X_{n}\right) \subseteq \mathcal{R}%
\left( T\right) ,\quad \underset{n\rightarrow \infty }{\mathrm{s\hspace{%
0.15em}}\!\text{-\negthinspace }\lim }P_{\mathcal{R}\left( T_{n}\right) }=P_{%
\overline{\mathcal{R}\left( T\right) }}, \\
\mathcal{N}\left( T_{n}\right) =\left( \mathcal{N}\left( T\right) \cap
X_{n}\right) \overset{\bot }{\oplus }X_{n}^{\,\bot },\quad P_{\mathcal{N}%
\left( T_{n}\right) }=P_{\mathcal{N}\left( T\right) \cap X_{n}}+I-P_{X_{n}},%
\end{array}%
\right.  \label{eqn2.0_0}
\end{equation}%
and
\begin{equation}
T_{n}^{\dag }y-T^{\dag }y=\left( T_{n}^{\dag }T-I\right) \left(
I-P_{X_{n}}\right) T^{\dag }y\quad \forall \,y\in \mathcal{D}\left( T^{\dag
}\right) .  \label{eqn2.0}
\end{equation}

$\mathrm{(b)}$ \ There hold%
\begin{equation*}
\underset{n\rightarrow \infty }{\mathrm{s\hspace{0.15em}}\!\text{%
-\negthinspace }\lim }P_{\mathcal{N}\left( T_{n}\right) }=P_{\mathcal{N}%
\left( T\right) }\Longleftrightarrow \,\underset{n\rightarrow \infty }{%
\mathrm{s\hspace{0.15em}}\!\text{-\negthinspace }\lim }P_{\mathcal{N}\left(
T\right) \cap X_{n}}=P_{\mathcal{N}\left( T\right) }\Longleftrightarrow
\text{(\ref{eqn1.4})},
\end{equation*}%
and%
\begin{equation*}
\text{(\ref{eqn1.1})}\Longleftrightarrow \text{(\ref{eqn1.2})}%
\Longleftrightarrow \text{(\ref{eqn1.3})}\Longrightarrow \text{(\ref{eqn1.4}%
).}
\end{equation*}

$\mathrm{(c)}$ \ If (\ref{eqn1.4}) and $\dim \mathcal{N}\left( T\right)
<\infty $ are valid, then there is a $n_{\ast }\in
\mathbb{N}
$ such that
\begin{equation}
\left.
\begin{array}{r}
X_{n}\supseteq \mathcal{N}\left( T\right) , \\
\mathcal{N}\left( T_{n}\right) =\mathcal{N}\left( T\right) \overset{\bot }{%
\oplus }X_{n}^{\,\bot }%
\end{array}%
\right\} \text{\quad for }n\geq n_{\ast }\text{,}  \label{eqn2.1}
\end{equation}
\end{lemma}

\begin{proof}
\textrm{(a)} \ It is clear that%
\begin{equation*}
\mathcal{R}\left( T_{n}\right) =T\left( X_{n}\right) \subseteq \mathcal{R}%
\left( T\right) ,
\end{equation*}%
and%
\begin{equation*}
\mathcal{R}\left( T\right) \subseteq \,\underset{n\rightarrow \infty }{%
\mathrm{s\hspace{0.15em}}\text{-}\lim }\mathcal{R}\left( T_{n}\right)
\subseteq \,\underset{n\rightarrow \infty }{\mathrm{w\hspace{0.15em}}\text{-}%
\widetilde{\lim }}\mathcal{R}\left( T_{n}\right) \subseteq \overline{%
\mathcal{R}\left( T\right) }.
\end{equation*}%
Since $\mathrm{s\hspace{0.15em}}$-$\lim_{n\rightarrow \infty }\mathcal{R}%
\left( T_{n}\right) $ is closed by \cite[Lemma 2.7]{nailin2005basic}, we have%
\begin{equation*}
\underset{n\rightarrow \infty }{\mathrm{s\hspace{0.15em}}\text{-}\lim }%
\mathcal{R}\left( T_{n}\right) =\,\underset{n\rightarrow \infty }{\mathrm{w%
\hspace{0.15em}}\text{-}\widetilde{\lim }}\mathcal{R}\left( T_{n}\right) =%
\overline{\mathcal{R}\left( T\right) }.
\end{equation*}%
By Lemma \ref{lem2.0_1} that is equivalent to%
\begin{equation}
\underset{n\rightarrow \infty }{\mathrm{s\hspace{0.15em}}\!\text{%
-\negthinspace }\lim }P_{\mathcal{R}\left( T_{n}\right) }=P_{\overline{%
\mathcal{R}\left( T\right) }}.  \label{eqn2.1_0}
\end{equation}

It is clear that%
\begin{equation*}
\mathcal{N}\left( T_{n}\right) =\left( \mathcal{N}\left( T\right) \cap
X_{n}\right) \overset{\bot }{\oplus }X_{n}^{\,\bot },\medskip \quad P_{%
\mathcal{N}\left( T_{n}\right) }=P_{\mathcal{N}\left( T\right) \cap
X_{n}}+I-P_{X_{n}},
\end{equation*}%
and hence for all $y\in \mathcal{D}\left( T^{\dag }\right) $%
\begin{eqnarray*}
T_{n}^{\dag }y-T^{\dag }y &=&T_{n}^{\dag }\left( TT^{\dag
}-TP_{X_{n}}T^{\dag }\right) y+\left( T_{n}^{\dag }T_{n}-T^{\dag }T\right)
T^{\dag }y \\
&=&T_{n}^{\dag }T\left( I-P_{X_{n}}\right) T^{\dag }y+\left( P_{\mathcal{N}%
\left( T_{n}\right) ^{\bot }}-P_{\mathcal{N}\left( T\right) ^{\bot }}\right)
T^{\dag }y \\
&=&T_{n}^{\dag }T\left( I-P_{X_{n}}\right) T^{\dag }y+\left(
P_{X_{n}}-I\right) T^{\dag }y \\
&=&\left( T_{n}^{\dag }T-I\right) \left( I-P_{X_{n}}\right) T^{\dag }y.
\end{eqnarray*}

\textrm{(b)} \ It follows from (\ref{eqn2.0_0}) that%
\begin{equation*}
\underset{n\rightarrow \infty }{\mathrm{s\hspace{0.15em}}\!\text{%
-\negthinspace }\lim }P_{\mathcal{N}\left( T_{n}\right) }=P_{\mathcal{N}%
\left( T\right) }\Longleftrightarrow \underset{n\rightarrow \infty }{\mathrm{%
s\hspace{0.15em}}\!\text{-\negthinspace }\lim }P_{\mathcal{N}\left( T\right)
\cap X_{n}}=P_{\mathcal{N}\left( T\right) }.
\end{equation*}%
Note that
\begin{equation*}
\underset{n\rightarrow \infty }{\mathrm{s\hspace{0.15em}}\text{-}\lim }%
\left( \mathcal{N}\left( T\right) \cap X_{n}\right) =\left\{ x\in
X\,:\lim_{n\rightarrow \infty }\mathrm{dist\hspace{0.15em}}\left( x,\mathcal{%
N}\left( T\right) \cap X_{n}\right) =0\right\}
\end{equation*}%
and%
\begin{equation*}
\underset{n\rightarrow \infty }{\mathrm{s\hspace{0.15em}}\text{-}\lim }%
\left( \mathcal{N}\left( T\right) \cap X_{n}\right) \subseteq \underset{%
n\rightarrow \infty }{\mathrm{w\hspace{0.15em}}\text{-}\widetilde{\lim }}%
\left( \mathcal{N}\left( T\right) \cap X_{n}\right) \subseteq \mathcal{N}%
\left( T\right) .
\end{equation*}%
Then, by Lemma \ref{lem2.0_1},%
\begin{eqnarray*}
\text{(\ref{eqn1.4})} &\Longleftrightarrow &\underset{n\rightarrow \infty }{%
\mathrm{s\hspace{0.15em}}\text{-}\lim }\left( \mathcal{N}\left( T\right)
\cap X_{n}\right) =\underset{n\rightarrow \infty }{\mathrm{w\hspace{0.15em}}%
\text{-}\widetilde{\lim }}\left( \mathcal{N}\left( T\right) \cap
X_{n}\right) =\mathcal{N}\left( T\right) . \\
&\Longleftrightarrow &\underset{n\rightarrow \infty }{\mathrm{s\hspace{0.15em%
}}\!\text{-\negthinspace }\lim }P_{\mathcal{N}\left( T\right) \cap X_{n}}=P_{%
\mathcal{N}\left( T\right) } \\
&\Longleftrightarrow &\underset{n\rightarrow \infty }{\mathrm{s\hspace{0.15em%
}}\!\text{-\negthinspace }\lim }P_{\mathcal{N}\left( T_{n}\right) }=P_{%
\mathcal{N}\left( T\right) }.
\end{eqnarray*}

Clearly, by the uniform boundedness principle and (\ref{eqn2.0}) we have that%
\begin{equation*}
\text{(\ref{eqn1.2})}\Longleftrightarrow \text{(\ref{eqn1.1})}%
\Longleftrightarrow \text{(\ref{eqn1.3}).}
\end{equation*}%
From (\ref{eqn2.0}) it follows that%
\begin{equation*}
T_{n}^{\dag }T_{n}-T^{\dag }TP_{X_{n}}=\left( T_{n}^{\dag }T-I\right) \left(
I-P_{X_{n}}\right) T^{\dag }TP_{X_{n}},
\end{equation*}%
so, by (\ref{eqn2.0_0}), there is%
\begin{equation*}
P_{\mathcal{N}\left( T\right) \cap X_{n}}-P_{\mathcal{N}\left( T\right)
}P_{X_{n}}=\left( T_{n}^{\dag }T-I\right) P_{X_{n}^{\bot }}P_{\mathcal{N}%
\left( T\right) }\left( I-P_{X_{n}^{\bot }}\right) .
\end{equation*}%
Hence, (\ref{eqn1.3}) implies that%
\begin{equation*}
P_{\mathcal{N}\left( T\right) \cap X_{n}}\overset{\mathrm{s\hspace{0.15em}}}{%
\rightarrow }P_{\mathcal{N}\left( T\right) }\ \left( n\rightarrow \infty
\right) ,\quad \text{i.e.,\quad (\ref{eqn1.4}) holds.}
\end{equation*}

\textrm{(c)} \ Since (\ref{eqn1.4}) is equivalent to%
\begin{equation*}
\underset{n\rightarrow \infty }{\mathrm{s\hspace{0.15em}}\!\text{%
-\negthinspace }\lim }P_{\mathcal{N}\left( T\right) \cap X_{n}}=P_{\mathcal{N%
}\left( T\right) },
\end{equation*}%
and since
\begin{equation*}
{{\mathcal{N}(T)}\cap X_{n}\subseteq {\mathcal{N}(T),\quad }}\dim \left( {{%
\mathcal{N}(T)}\cap X_{n}}\right) \leq \dim \mathcal{N}\left( T\right)
<\infty ,
\end{equation*}%
there holds%
\begin{equation*}
\underset{n\rightarrow \infty }{\text{\negthinspace }\lim }\left\Vert P_{%
\mathcal{N}\left( T\right) \cap X_{n}}-P_{\mathcal{N}\left( T\right)
}\right\Vert =0.
\end{equation*}%
Thus, there is a $n_{\ast }\in
\mathbb{N}
$ such that%
\begin{equation*}
\left\Vert P_{\mathcal{N}\left( T\right) \cap X_{n}}-P_{\mathcal{N}\left(
T\right) }\right\Vert <1\quad \text{for }n\geq n_{\ast }.
\end{equation*}%
Note that this implies that $\mathrm{rank\hspace{0.15em}}P_{\mathcal{N}%
\left( T\right) \cap X_{n}}=\mathrm{rank\hspace{0.15em}}P_{\mathcal{N}\left(
T\right) }$ by \cite[Theorem 2.3]{stewart1977perturbation} (in fact, $P_{%
\mathcal{N}\left( T\right) \cap X_{n}}$ and $P_{\mathcal{N}\left( T\right) }$
are unitarily equivalent by \cite[p.56, Theorem 6.32]{Kato1980}), and
therefore
\begin{equation*}
\dim \left( {{\mathcal{N}(T)}\cap X_{n}}\right) =\dim \mathcal{N}\left(
T\right) \quad \text{for }n\geq n_{\ast }.
\end{equation*}%
So, we have
\begin{equation*}
{X}_{n}\supseteq {{\mathcal{N}(T)}\cap X_{n}={\mathcal{N}(T)}}\quad \text{%
for }n\geq n_{\ast },
\end{equation*}%
and therefore%
\begin{equation*}
\mathcal{N}\left( T_{n}\right) =\left( \mathcal{N}\left( T\right) \cap
X_{n}\right) \overset{\bot }{\oplus }X_{n}^{\,\bot }=\mathcal{N}\left(
T\right) \overset{\bot }{\oplus }X_{n}^{\,\bot }\quad \text{for }n\geq
n_{\ast }.
\end{equation*}%
Thus, we obtain (\ref{eqn2.1}).
\end{proof}

\begin{lemma}
\label{lem2.1} Let $T\in \mathcal{B}\left( X,Y\right) $ have LPA $\left\{
\left( X_{n},T_{n}\right) \right\} $. Then%
\begin{equation*}
\underset{n\rightarrow \infty }{\mathrm{s\hspace{0.15em}}\text{-}\lim }%
\mathcal{G}\left( T_{n}\right) =\underset{n\rightarrow \infty }{\mathrm{w%
\hspace{0.15em}}\text{-}\widetilde{\lim }}\mathcal{G}\left( T_{n}\right) =%
\mathcal{G}\left( T\right) .
\end{equation*}
\end{lemma}

\begin{proof}
It is clear that $\mathrm{s\hspace{0.15em}}$-$\lim_{n\rightarrow \infty }%
\mathcal{G}\left( T_{n}\right) \subseteq \mathrm{w\hspace{0.15em}}$-$%
\widetilde{\lim }_{n\rightarrow \infty }\mathcal{G}\left( T_{n}\right) $.
Hence, we need only to show that%
\begin{equation*}
\underset{n\rightarrow \infty }{\mathrm{w\hspace{0.15em}}\text{-}\widetilde{%
\lim }}\mathcal{G}\left( T_{n}\right) \subseteq \mathcal{G}\left( T\right)
\subseteq \underset{n\rightarrow \infty }{\mathrm{s\hspace{0.15em}}\text{-}%
\lim }\mathcal{G}\left( T_{n}\right) .
\end{equation*}%
Let $\left( x,y\right) \in \mathcal{G}\left( T\right) $. Then $\left(
x,T_{n}x\right) \in \mathcal{G}\left( T_{n}\right) $, and $\left(
x,T_{n}x\right) \rightarrow \left( x,y\right) \ \left( n\rightarrow \infty
\right) $. Therefore, $\left( x,y\right) \in \mathrm{s\hspace{0.15em}}$-$%
\lim_{n\rightarrow \infty }\mathcal{G}\left( T_{n}\right) $. This gives that
\begin{equation*}
\mathcal{G}\left( T\right) \subseteq \underset{n\rightarrow \infty }{\mathrm{%
s\hspace{0.15em}}\text{-}\lim }\mathcal{G}\left( T_{n}\right) .
\end{equation*}%
Let $\left( x,y\right) \in \mathrm{w\hspace{0.15em}}$-$\widetilde{\lim }%
_{n\rightarrow \infty }\mathcal{G}\left( T_{n}\right) $. Then there is a
sequence $\left\{ \left( x_{n},y_{n}\right) \right\} $ such that $\cup
_{k=n}^{\infty }\mathcal{G}\left( T_{k}\right) \ni \left( x_{n},y_{n}\right)
\rightharpoonup \left( x,y\right) \ \left( n\rightarrow \infty \right) $.
Thus, there is a sequence $\left\{ k_{n}\right\} $ such that
\begin{equation*}
k_{n}\geq n,\quad x_{n}\rightharpoonup x\ \left( n\rightarrow \infty \right)
,\quad T_{k_{n}}x_{n}=y_{n}\rightharpoonup y\ \left( n\rightarrow \infty
\right) .
\end{equation*}%
Note that for all $v\in Y$ there holds%
\begin{eqnarray*}
\left\vert \left\langle Tx-y,v\right\rangle \right\vert &\leq &\left\vert
\left\langle T\left( x-x_{n}\right) ,v\right\rangle \right\vert +\left\vert
\left\langle \left( T-T_{k_{n}}\right) x_{n},v\right\rangle \right\vert
+\left\vert \left\langle y_{n}-y,v\right\rangle \right\vert \\
&=&\left\vert \left\langle x-x_{n},T^{\ast }v\right\rangle \right\vert
+\left\vert \left\langle x_{n},\left( I-P_{X_{k_{n}}}\right) T^{\ast
}v\right\rangle \right\vert +\left\vert \left\langle y_{n}-y,v\right\rangle
\right\vert .
\end{eqnarray*}%
Thus, we have
\begin{equation*}
\left\langle Tx-y,v\right\rangle =0\ \forall v\in X,\text{\quad that
is,\quad }\left( x,y\right) \in \mathcal{G}\left( T\right) .
\end{equation*}%
This gives that
\begin{equation*}
\underset{n\rightarrow \infty }{\mathrm{w\hspace{0.15em}}\text{-}\widetilde{%
\lim }}\mathcal{G}\left( T_{n}\right) \subseteq \mathcal{G}\left( T\right) .
\end{equation*}
\end{proof}

Now, we can prove Theorem \ref{the1.0} as follows.

\begin{proof}[Proof of Theorem \protect\ref{the1.0}]
By statement (b) of Lemma \ref{lem2.0} and Lemma \ref{lem2.0_1}, we have%
\begin{equation*}
\text{(\ref{eqn1.4})}\Longleftrightarrow \,\underset{n\rightarrow \infty }{%
\mathrm{s\hspace{0.15em}}\text{-}\lim }\mathcal{N}\left( T_{n}\right) =%
\underset{n\rightarrow \infty }{\mathrm{w\hspace{0.15em}}\text{-}\widetilde{%
\lim }}\mathcal{N}\left( T_{n}\right) =\mathcal{N}\left( T\right) \text{.}
\end{equation*}

\textrm{(a) }$\Longrightarrow $\textrm{\ (b)}: \ Let (a) hold. Then, by
Lemma \ref{lem2.0_1}, we have%
\begin{equation}
\underset{n\rightarrow \infty }{\text{s-}\lim }P_{\mathcal{N}\left(
T_{n}\right) ^{\bot }}=P_{\mathcal{N}\left( T\right) ^{\bot }}.
\label{eqn2.2_0}
\end{equation}%
By statement (a) of Lemma \ref{lem2.0},%
\begin{equation}
\underset{n\rightarrow \infty }{\mathrm{s\hspace{0.15em}}\!\text{%
-\negthinspace }\lim }P_{\mathcal{R}\left( T_{n}\right) ^{\bot }}=P_{%
\mathcal{R}\left( T\right) ^{\bot }}.  \label{eqn2.2_1}
\end{equation}%
To prove (b) we need only to show that%
\begin{equation}
\underset{n\rightarrow \infty }{\mathrm{w\hspace{0.15em}}\text{-}\widetilde{%
\lim }}\mathcal{G}\left( T_{n}^{\,\dag }\right) \subseteq \mathcal{G}\left(
T^{\dag }\right) \subseteq \underset{n\rightarrow \infty }{\mathrm{s\hspace{%
0.15em}}\text{-}\lim }\mathcal{G}\left( T_{n}^{\,\dag }\right) .
\label{eqn2.2}
\end{equation}%
Let $\left( y,x\right) \in \mathcal{G}\left( T^{\dag }\right) $. Then
\begin{equation*}
\left( T^{\dag }y,P_{\overline{\mathcal{R}\left( T\right) }}y\right) =\left(
T^{\dag }y,TT^{\dag }y\right) =\left( x,Tx\right) \in \mathcal{G}\left(
T\right) .
\end{equation*}%
By Lemma \ref{lem2.1}, $\mathrm{s\hspace{0.15em}}$-$\lim_{n\rightarrow
\infty }\mathcal{G}\left( T_{n}\right) =\mathcal{G}\left( T\right) $, hence
there is a sequence $\left\{ x_{n}\right\} $ such that
\begin{equation*}
x_{n}\rightarrow x\ \left( n\rightarrow \infty \right) ,\quad
T_{n}x_{n}\rightarrow Tx\ \left( n\rightarrow \infty \right) .
\end{equation*}%
This with (\ref{eqn2.2_0}) and (\ref{eqn2.2_1}) implies that%
\begin{equation*}
\left\{
\begin{array}{l}
T_{n}x_{n}\rightarrow Tx\ \left( n\rightarrow \infty \right) , \\
P_{\mathcal{N}\left( T_{n}\right) ^{\bot }}x_{n}\rightarrow P_{\mathcal{N}%
\left( T\right) ^{\bot }}x\ \left( n\rightarrow \infty \right) , \\
P_{\mathcal{R}\left( T_{n}\right) ^{\bot }}y\rightarrow P_{\mathcal{R}\left(
T\right) ^{\bot }}y\ \left( n\rightarrow \infty \right) ,%
\end{array}%
\right.
\end{equation*}%
and therefore
\begin{equation*}
\left\{
\begin{array}{l}
T_{n}x_{n}+P_{\mathcal{R}\left( T_{n}\right) ^{\bot }}y\rightarrow Tx+P_{%
\mathcal{R}\left( T\right) ^{\bot }}y=y\ \left( n\rightarrow \infty \right) ,
\\
T_{n}^{\dag }\left( T_{n}x_{n}+P_{\mathcal{R}\left( T_{n}\right) ^{\bot
}}y\right) \rightarrow T^{\dag }Tx=T^{\dag }y=x\ \left( n\rightarrow \infty
\right) .%
\end{array}%
\right.
\end{equation*}%
Thus, $\left( y,x\right) \in \mathrm{s\hspace{0.15em}}$-$\lim_{n\rightarrow
\infty }\mathcal{G}\left( T_{n}^{\dag }\right) $, we get%
\begin{equation*}
\mathcal{G}\left( T^{\dag }\right) \subseteq \,\underset{n\rightarrow \infty
}{\mathrm{s\hspace{0.15em}}\text{-}\lim }\mathcal{G}\left( T_{n}^{\,\dag
}\right) .
\end{equation*}%
Let $\left( y,x\right) \in \mathrm{w\hspace{0.15em}}$-$\widetilde{\lim }%
_{n\rightarrow \infty }\mathcal{G}\left( T_{n}^{\,\dag }\right) $. Then
there is a sequence $\left\{ \left( y_{n},x_{n}\right) \right\} $ such that%
\begin{equation*}
\underset{k=n}{\overset{\infty }{\cup }}\mathcal{G}\left( T_{k}^{\,\dag
}\right) \ni \left( y_{n},x_{n}\right) \rightharpoonup \left( y,x\right) \
\left( n\rightarrow \infty \right) .
\end{equation*}%
Hence, there is a sequence $\left\{ k_{n}\right\} $ such that
\begin{equation}
k_{n}\geq n,\quad y_{n}\rightharpoonup y\ \left( n\rightarrow \infty \right)
,\quad T_{k_{n}}^{\,\dag }y_{n}=x_{n}\rightharpoonup x\ \left( n\rightarrow
\infty \right) .  \label{eqn2.2_2}
\end{equation}%
By (\ref{eqn2.0_0}) in statement (a) of Lemma \ref{lem2.0}, we have (\ref%
{eqn2.1_0}) and $T_{k_{n}}^{\,\dag }y_{n}\in \mathcal{N}\left(
T_{k_{n}}\right) ^{\bot }\subseteq X_{k_{n}}$. From (\ref{eqn2.2_2}), (\ref%
{eqn2.2_0}), by Lemma \ref{lem2.0_1}, we have%
\begin{equation*}
x\in \underset{n\rightarrow \infty }{\mathrm{w\hspace{0.15em}}\text{-}%
\widetilde{\lim }}\mathcal{N}\left( T_{n}\right) ^{\bot }=\mathcal{N}\left(
T\right) ^{\bot },
\end{equation*}%
\begin{equation*}
T_{k_{n}}x_{n}=T_{k_{n}}T_{k_{n}}^{\,\dag }y_{n}=TT_{k_{n}}^{\,\dag
}y_{n}\rightharpoonup Tx\ \left( n\rightarrow \infty \right) ,
\end{equation*}%
and for any $v\in X$%
\begin{eqnarray*}
\left\langle T_{k_{n}}x_{n},v\right\rangle &=&\left\langle P_{\mathcal{R}%
\left( T_{k_{n}}\right) }y_{n},v\right\rangle =\left\langle y_{n},P_{%
\mathcal{R}\left( T_{k_{n}}\right) }v\right\rangle \\
&=&\left\langle y_{n},\left( P_{\mathcal{R}\left( T_{k_{n}}\right) }-P_{%
\overline{\mathcal{R}\left( T\right) }}\right) v\right\rangle +\left\langle
y_{n},P_{\overline{\mathcal{R}\left( T\right) }}v\right\rangle \\
&\rightarrow &\left\langle y,P_{\overline{\mathcal{R}\left( T\right) }%
}v\right\rangle =\left\langle P_{\overline{\mathcal{R}\left( T\right) }%
}y,v\right\rangle \ \left( n\rightarrow \infty \right) .
\end{eqnarray*}%
Thus,%
\begin{equation*}
x\in \mathcal{N}\left( T\right) ^{\bot },\ Tx=P_{\overline{\mathcal{R}\left(
T\right) }}y,\quad \text{that is,\quad }\left( y,x\right) \in \mathcal{G}%
\left( T^{\dag }\right) .
\end{equation*}%
So, we obtain that%
\begin{equation*}
\underset{n\rightarrow \infty }{\mathrm{w\hspace{0.15em}}\text{-}\widetilde{%
\lim }}\mathcal{G}\left( T_{n}^{\,\dag }\right) \subseteq \mathcal{G}\left(
T^{\dag }\right) .
\end{equation*}%
Now, (\ref{eqn2.2}) is proved.

\textrm{(b) }$\Longrightarrow $\textrm{\ (c)}: \ Let (b) hold. To prove (c)
let $\left\{ y_{n}\right\} \subseteq Y$ satisfy%
\begin{equation*}
\sup_{n}\left\Vert T_{n}^{\dag }y_{n}\right\Vert <+\infty \text{\quad
and\quad }\underset{n\rightarrow \infty }{\mathrm{w\hspace{0.15em}}\text{-}%
\lim }y_{n}=y.
\end{equation*}%
Then any subsequence $\left\{ T_{n_{k}}^{\dag }y_{n_{k}}\right\} $ of $%
\left\{ T_{n}^{\dag }y_{n}\right\} $ has a subsequence, again denoted by $%
\left\{ T_{n_{k}}^{\dag }y_{n_{k}}\right\} $, converging weakly to some $%
u\in X$. By use of (b) we have that%
\begin{equation*}
\left( y,u\right) \in \underset{n\rightarrow \infty }{\mathrm{w\hspace{0.15em%
}}\text{-}\widetilde{\lim }}\mathcal{G}\left( T_{n}^{\,\dag }\right) =%
\mathcal{G}\left( T^{\dag }\right) .
\end{equation*}%
This gives that%
\begin{equation*}
y\in \mathcal{D}\left( T^{\dag }\right) ,\quad T_{n_{k}}^{\dag
}y_{n_{k}}\rightharpoonup u=T^{\dag }y\quad \left( k\rightarrow \infty
\right)
\end{equation*}%
So, every subsequence of $\left\{ T_{n}^{\dag }y_{n}\right\} $ has a
subsequence converging weakly to $T^{\dag }y$ and hence%
\begin{equation*}
T_{n}^{\dag }y_{n}\rightharpoonup T^{\dag }y\quad \left( n\rightarrow \infty
\right) .
\end{equation*}%
Thus, we obtain (c).

\textrm{(c) }$\Longrightarrow $\textrm{\ (a)}: \ Let (c) hold.\ To prove
(a), by statement (b) of Lemma \ref{lem2.0} we need only to show that%
\begin{equation}
\underset{n\rightarrow \infty }{\mathrm{s\hspace{0.15em}}\text{-}\lim }P_{%
\mathcal{N}\left( T_{n}\right) }=P_{\mathcal{N}\left( T\right) }.
\label{eqn2.2_3}
\end{equation}%
In fact, for any $x\in X$ let%
\begin{equation*}
y_{n}:=T_{n}x\text{ }\left( n\in
\mathbb{N}
\right) ,
\end{equation*}%
then%
\begin{equation*}
y_{n}\rightarrow Tx\ \left( n\rightarrow \infty \right) \quad \text{and\quad
}\sup_{n}\left\Vert T_{n}^{\dag }y_{n}\right\Vert \leq \left\Vert
x\right\Vert .
\end{equation*}%
Due to (c), it follows that%
\begin{equation*}
T_{n}^{\dag }T_{n}x=T_{n}^{\dag }y_{n}\rightharpoonup T^{\dag }Tx\quad
\left( n\rightarrow \infty \right) .
\end{equation*}%
Hence we have that%
\begin{equation*}
\underset{n\rightarrow \infty }{\mathrm{w\hspace{0.15em}}\text{-}\lim }P_{%
\mathcal{N}\left( T_{n}\right) }=P_{\mathcal{N}\left( T\right) }.
\end{equation*}%
This is equivalent to (\ref{eqn2.2_3}) by Lemma \ref{lem2.0_1}.
\end{proof}

\section{Proof of Theorem \protect\ref{the1.1}}

To prove Theorem \ref{the1.1}, we need more lemmas besides Lemma \ref{lem2.0}%
.

\begin{lemma}
\label{lem3.0_1} If $P,Q$ are orthogonal-projections on a Hilbert space $H$,
then%
\begin{equation*}
\left\Vert P-Q\right\Vert =\max \left\{ \left\Vert (I-Q)P\right\Vert
,\left\Vert (I-P)Q\right\Vert \right\} \leq 1,
\end{equation*}%
and if $\left\Vert P-Q\right\Vert <1$, there holds%
\begin{equation*}
\left\Vert P-Q\right\Vert =\left\Vert \left( I-Q\right) P\right\Vert
=\left\Vert \left( I-P\right) Q\right\Vert .
\end{equation*}
\end{lemma}

\begin{proof}
Since $(I-Q)P=(P-Q)P,\ (I-P)Q=(Q-P)Q$, it follows that
\begin{equation*}
\left\Vert \left( I-Q\right) P\right\Vert \leq \left\Vert P-Q\right\Vert
,\qquad \left\Vert \left( I-P\right) Q\right\Vert \leq \left\Vert
P-Q\right\Vert ,
\end{equation*}%
and therefore%
\begin{equation*}
\max \left\{ \Vert (I-Q)P\Vert ,\Vert (I-P)Q\Vert \right\} \leq \left\Vert
P-Q\right\Vert .
\end{equation*}%
Note that for each $x\in H$ there hold%
\begin{eqnarray*}
{\Vert (P-Q)x\Vert }^{2} &=&{\Vert (I-Q)Px\Vert }^{2}+{\Vert Q(I-P)x\Vert }%
^{2} \\
&\leq &{\Vert (I-Q)P\Vert }^{2}{\Vert Px\Vert }^{2}+{\Vert Q(I-P)\Vert }^{2}{%
\Vert (I-P)x\Vert }^{2} \\
&=&{\Vert (I-Q)P\Vert }^{2}{\Vert Px\Vert }^{2}+{\Vert {(I-P)}^{\ast
}Q^{\ast }\Vert }^{2}{\Vert (I-P)x\Vert }^{2} \\
&\leq &{(\max \{\Vert (I-Q)P\Vert ,\Vert (I-P)Q\Vert \})}^{2}{\Vert x\Vert }%
^{2},
\end{eqnarray*}%
that is%
\begin{equation*}
\left\Vert P-Q\right\Vert \leq \max \left\{ \Vert (I-Q)P\Vert ,\Vert
(I-P)Q\Vert \right\}.
\end{equation*}%
Thus we obtain
\begin{equation*}
\left\Vert P-Q\right\Vert =\max \{\left\Vert \left( I-Q\right) P\right\Vert
,\left\Vert \left( I-P\right) Q\right\Vert \leq 1.
\end{equation*}

The rest follows from \cite[Theorem 6.34, pp.56-58]{Kato1980}.
\end{proof}

\begin{lemma}
\label{lem3.0_2} Let $M,N$ be two closed subspaces of a Hilbert space $H$.
Then%
\begin{equation*}
\delta \left( M,N\right) =\left\Vert \left( I-P_{N}\right) P_{M}\right\Vert
,\quad \limfunc{gap}\left( M,N\right) =\left\Vert P_{M}-P_{N}\right\Vert ,
\end{equation*}%
and%
\begin{equation*}
\limfunc{gap}\left( M,N\right) =\delta \left( M,N\right) =\delta \left(
N,M\right) \quad \text{if }\limfunc{gap}\left( M,N\right) <1.
\end{equation*}
\end{lemma}

\begin{proof}
If $M=0$, it is clear that%
\begin{equation*}
\delta \left( M,N\right) =0=\left\Vert \left( I-P_{N}\right)
P_{M}\right\Vert .
\end{equation*}%
Next, assume $M\neq \left\{ 0\right\} $. Then we have that%
\begin{eqnarray*}
\delta \left( M,N\right) &=&\sup \left\{ \mathrm{dist\hspace{0.15em}}\left(
x,N\right) :x\in M,\ \left\Vert x\right\Vert =1\right\} \\
&=&\sup \left\{ \left\Vert \left( I-P_{N}\right) x\right\Vert :x\in M,\
\left\Vert x\right\Vert =1\right\} \\
&=&\sup \left\{ \left\Vert \left( I-P_{N}\right) P_{M}x\right\Vert :x\in M,\
\left\Vert x\right\Vert =1\right\} \\
&\leq &\left\Vert \left( I-P_{N}\right) P_{M}\right\Vert ,
\end{eqnarray*}%
and that for $x\in H$ with $\left\Vert x\right\Vert =1$ there holds%
\begin{eqnarray*}
\left\Vert \left( I-P_{N}\right) P_{M}x\right\Vert &=&\mathrm{dist\hspace{%
0.15em}}\left( P_{M}x,N\right) \\
&\leq &\delta \left( M,N\right) \left\Vert P_{M}x\right\Vert \leq \delta
\left( M,N\right) ,
\end{eqnarray*}%
that is,%
\begin{equation*}
\left\Vert \left( I-P_{N}\right) P_{M}\right\Vert \leq \delta \left(
M,N\right) .
\end{equation*}%
Thus, we obtain
\begin{equation*}
\delta \left( M,N\right) =\left\Vert \left( I-P_{N}\right) P_{M}\right\Vert .
\end{equation*}%
This with Lemma \ref{lem3.0_1} gives that%
\begin{equation*}
\limfunc{gap}\left( M,N\right) =\left\Vert P_{M}-P_{N}\right\Vert ,
\end{equation*}%
and%
\begin{equation*}
\limfunc{gap}\left( M,N\right) =\delta \left( M,N\right) =\delta \left(
N,M\right) \quad \text{if }\limfunc{gap}\left( M,N\right) <1.
\end{equation*}
\end{proof}

\begin{lemma}
\label{lem3.0} Let $S$ be an oblique-projection on a Hilbert space $H$ ($%
S^{2}=S\in \mathcal{B}\left( H\right) $). Then%
\begin{equation*}
\left\Vert P_{\mathcal{R}(S)}-P_{\mathcal{R}(S^{\ast })}\right\Vert
=\left\Vert P_{\mathcal{R}(S)}P_{\mathcal{N}(S)}\right\Vert =\left\Vert P_{%
\mathcal{N}(S)}P_{\mathcal{R}(S)}\right\Vert ,
\end{equation*}%
and
\begin{equation*}
\left\Vert P_{\mathcal{N}(S)}P_{\mathcal{R}(S)}\right\Vert =\left\{
\begin{array}{rl}
\sqrt{1-{\Vert S\Vert }^{-2}}, & \text{as}\ S\neq 0, \\
0, & \text{as}\ S=0;%
\end{array}%
\right.
\end{equation*}
\end{lemma}

\begin{proof}
Let $x\in H$. Since $\mathcal{N}(S)=\mathcal{R}(I-S)$, we have%
\begin{eqnarray*}
\left\Vert x\right\Vert ^{2} &=&\left\Vert {Sx+P_{\mathcal{N}(S)}(I-S)x}%
\right\Vert ^{2} \\
&=&\left\Vert {P_{\mathcal{N}(S)}x}\right\Vert ^{2}+\left\Vert {(I-P_{%
\mathcal{N}(S)})Sx}\right\Vert ^{2} \\
&=&{\Vert P_{\mathcal{N}(S)}x\Vert }^{2}+{\Vert Sx\Vert }^{2}-{\Vert P_{%
\mathcal{N}(S)}Sx\Vert }^{2} \\
&=&{\Vert P_{\mathcal{N}(S)}x\Vert }^{2}+{\Vert Sx\Vert }^{2}-{\Vert P_{%
\mathcal{N}(S)}P_{\mathcal{R}(S)}Sx\Vert }^{2} \\
&\geqslant &{\Vert P_{\mathcal{N}(S)}x\Vert }^{2}+{\Vert Sx\Vert }^{2}\left(
1-{\Vert P_{\mathcal{N}(S)}P_{\mathcal{R}(S)}\Vert }^{2}\right) ,
\end{eqnarray*}%
and therefore
\begin{equation*}
{\Vert (I-P_{\mathcal{N}(S)})x\Vert }^{2}={\Vert x\Vert }^{2}-{\Vert P_{%
\mathcal{N}(S)}x\Vert }^{2}\geq (1-{\Vert P_{\mathcal{N}(S)}P_{\mathcal{R}%
(S)}\Vert }^{2}){\Vert Sx\Vert }^{2}.
\end{equation*}%
This implies
\begin{equation*}
{\Vert I-P_{\mathcal{N}(S)}\Vert }^{2}\geq (1-{\Vert P_{\mathcal{N}(S)}P_{%
\mathcal{R}(S)}\Vert }^{2}){\Vert S\Vert }^{2}.
\end{equation*}%
So, if $S\neq 0$, there holds
\begin{equation*}
1={\Vert I-P_{\mathcal{N}(S)}\Vert }^{2}\geq {\Vert S\Vert }^{2}(1-{\Vert P_{%
\mathcal{N}(S)}P_{\mathcal{R}(S)}\Vert }^{2}),
\end{equation*}%
namely
\begin{equation*}
\left\Vert P_{\mathcal{N}(S)}P_{\mathcal{R}(S)}\right\Vert \geq \sqrt{%
1-\left\Vert {S}\right\Vert ^{-2}}.
\end{equation*}%
To prove the reverse inequality, observe that for each $x\in H$ we have%
\begin{eqnarray*}
{\Vert P_{\mathcal{R}(S)}x\Vert }^{2} &=&{\Vert SP_{\mathcal{R}(S)}x-SP_{%
\mathcal{N}(S)}P_{\mathcal{R}(S)}x\Vert }^{2} \\
&\leq &{\Vert S\Vert }^{2}{\Vert P_{\mathcal{R}(S)}x-P_{\mathcal{N}(S)}P_{%
\mathcal{R}(S)}x\Vert }^{2} \\
&=&{\Vert S\Vert }^{2}\left( {\Vert P_{\mathcal{R}(S)}x\Vert }^{2}-{\Vert P_{%
\mathcal{N}(S)}P_{\mathcal{R}(S)}x\Vert }^{2}\right) .
\end{eqnarray*}%
Hence, if $S\neq 0$, we have
\begin{equation*}
{\Vert P_{\mathcal{N}(S)}P_{\mathcal{R}(S)}x\Vert }^{2}\leq (1-{\Vert S\Vert
}^{-2}){\Vert P_{\mathcal{R}(S)}x\Vert }^{2}\leq (1-{\Vert S\Vert }^{-2}){%
\Vert x\Vert }^{2},
\end{equation*}%
namely
\begin{equation*}
\left\Vert P_{\mathcal{N}(S)}P_{\mathcal{R}(S)}\right\Vert \leq \sqrt{1-{%
\Vert S\Vert }^{-2}}.
\end{equation*}%
Thus, if $S\neq 0$, there holds%
\begin{equation*}
\left\Vert P_{\mathcal{N}(S)}P_{\mathcal{R}(S)}\right\Vert =\sqrt{1-{\Vert
S\Vert }^{-2}}<1.
\end{equation*}%
Now, if $S\neq 0$, by Lemma \ref{lem3.0_1}, we get
\begin{eqnarray*}
\left\Vert P_{\mathcal{R}(S)}-P_{\mathcal{R}(S^{\ast })}\right\Vert &=&\max
\left\{ \left\Vert \left( I-P_{\mathcal{R}(S^{\ast })}\right) P_{\mathcal{R}%
(S)}\right\Vert ,\left\Vert \left( I-P_{\mathcal{R}(S)}\right) P_{\mathcal{R}%
(S^{\ast })}\right\Vert \right\} \\
&=&\max \left\{ \left\Vert P_{\mathcal{N}(S)}P_{\mathcal{R}(S)}\right\Vert
,\left\Vert P_{\mathcal{N}(S^{\ast })}P_{\mathcal{R}(S^{\ast })}\right\Vert
\right\} \\
&=&\sqrt{1-{\Vert S\Vert }^{-2}}=\sqrt{1-{\Vert S}^{\ast }{\Vert }^{-2}}<1,
\end{eqnarray*}%
and therefore%
\begin{equation*}
\left\Vert P_{\mathcal{R}(S)}-P_{\mathcal{R}(S^{\ast })}\right\Vert
=\left\Vert P_{\mathcal{N}(S)}P_{\mathcal{R}(S)}\right\Vert =\left\Vert P_{%
\mathcal{N}(S^{\ast })}P_{\mathcal{R}(S^{\ast })}\right\Vert .
\end{equation*}%
Note that%
\begin{equation*}
\left\Vert P_{\mathcal{R}(S)}P_{\mathcal{N}(S)}\right\Vert =\left\Vert
\left( P_{\mathcal{N}(S)}P_{\mathcal{R}(S)}\right) ^{\ast }\right\Vert
=\left\Vert P_{\mathcal{N}(S)}P_{\mathcal{R}(S)}\right\Vert .
\end{equation*}%
So, it follows that%
\begin{equation*}
\left\Vert P_{\mathcal{R}(S)}-P_{\mathcal{R}(S^{\ast })}\right\Vert
=\left\Vert P_{\mathcal{R}(S)}P_{\mathcal{N}(S)}\right\Vert =\left\Vert P_{%
\mathcal{N}(S)}P_{\mathcal{R}(S)}\right\Vert ,\quad \text{and}
\end{equation*}%
\begin{equation*}
\left\Vert P_{\mathcal{N}(S)}P_{\mathcal{R}(S)}\right\Vert =\left\{
\begin{array}{rl}
\sqrt{1-{\Vert S\Vert }^{-2}}, & \text{as}\ S\neq 0, \\
0, & \text{as}\ S=0.%
\end{array}%
\right.
\end{equation*}
\end{proof}

\begin{lemma}
\label{lem3.1} Let $T\in \mathcal{B}\left( X,Y\right) $ have LPA $%
\{(X_{n},T_{n})\}_{n\in
\mathbb{N}
}$, $\left\{ Q_{n}\right\} _{n\in
\mathbb{N}
}$ be defined as
\begin{equation}
Q_{n}:=T^{\dag }P_{T\left( X_{n}\right) }T\quad \forall n\in
\mathbb{N}
.  \label{eqn3.0}
\end{equation}%
Then $\left\{ Q_{n}\right\} $ is a sequence of oblique-projections in $%
\mathcal{B}\left( X\right) $ ($Q_{n}^{2}=Q_{n}\in \mathcal{B}\left( X\right)
$) which satisfies:%
\begin{equation}
\mathcal{N}\left( Q_{n}\right) =\left[ T^{\ast }T\left( X_{n}\right) \right]
^{\bot },\quad \mathcal{R}\left( Q_{n}\right) =T^{\dag }T\left( X_{n}\right)
,  \label{eqn3.0_1}
\end{equation}%
and%
\begin{equation}
\underset{n\rightarrow \infty }{\mathrm{s\hspace{0.15em}}\text{-}\lim }P_{%
\mathcal{N}\left( Q_{n}\right) }=P_{\mathcal{N}\left( T\right) },\quad
\underset{n\rightarrow \infty }{\mathrm{s\hspace{0.15em}}\text{-}\lim }P_{%
\mathcal{R}\left( Q_{n}\right) }=P_{\mathcal{N}\left( T\right) ^{\bot }}.
\label{eqn3.0_2}
\end{equation}
\end{lemma}

\begin{proof}
Since $T^{\dag }$ is a closed operator and $T\left( X_{n}\right) \subseteq
\mathcal{D}\left( T^{\dag }\right) $ ($\forall n$), due to the closed graph
theorem and $T\left( X_{n}\right) $ being closed in $Y$, we see that each $%
\left. T^{\dag }\right\vert _{T\left( X_{n}\right) }$ is a bound linear
operator. Hence, $Q_{n}:=T^{\dag }P_{T\left( X_{n}\right) }T\in \mathcal{B}%
\left( X\right) $, and
\begin{equation*}
Q_{n}^{2}=T^{\dag }P_{T\left( X_{n}\right) }P_{\overline{\mathcal{R}\left(
T\right) }}P_{T\left( X_{n}\right) }T=Q_{n}.
\end{equation*}%
It is clear that%
\begin{equation*}
\mathcal{N}\left( Q_{n}\right) =\mathcal{N}\left( P_{T\left( X_{n}\right)
}T\right) =\mathcal{R}\left( T^{\ast }P_{T\left( X_{n}\right) }\right)
^{\bot }=\left[ T^{\ast }T\left( X_{n}\right) \right] ^{\bot },\quad
\end{equation*}%
and%
\begin{equation*}
\mathcal{R}\left( Q_{n}\right) =T^{\dag }T\left( X_{n}\right) =P_{\mathcal{N}%
\left( T\right) ^{\bot }}\left( X_{n}\right) =P_{\overline{\mathcal{R}\left(
T^{\ast }\right) }}\left( X_{n}\right) .
\end{equation*}%
Thus, we have (\ref{eqn3.0_1}) and%
\begin{equation*}
P_{\mathcal{N}\left( Q_{n}\right) }=I-P_{T^{\ast }T\left( X_{n}\right)
},\quad P_{\mathcal{R}\left( Q_{n}\right) }=P_{P_{\overline{\mathcal{R}%
\left( T^{\ast }\right) }}\left( X_{n}\right) }.
\end{equation*}%
Note that (\ref{eqn2.1_0}) holds by statement (a) of Lemma \ref{lem2.0}.
Since $T^{\ast }T,P_{\overline{\mathcal{R}\left( T^{\ast }\right) }}\in
\mathcal{B}\left( X\right) $, by replacing $T$ with $T^{\ast }T$ or $P_{%
\overline{\mathcal{R}\left( T^{\ast }\right) }}$ in (\ref{eqn2.1_0}), it
follows that
\begin{equation*}
\underset{n\rightarrow \infty }{\mathrm{s\hspace{0.15em}}\!\text{%
-\negthinspace }\lim }P_{T^{\ast }T\left( X_{n}\right) }=P_{\overline{%
\mathcal{R}\left( T^{\ast }T\right) }}=P_{\mathcal{N}\left( T\right) ^{\bot
}}
\end{equation*}%
and%
\begin{equation*}
\underset{n\rightarrow \infty }{\mathrm{s\hspace{0.15em}}\!\text{%
-\negthinspace }\lim }P_{P_{\overline{\mathcal{R}\left( T^{\ast }\right) }%
}\left( X_{n}\right) }=P_{\overline{\mathcal{R}\left( T^{\ast }\right) }}=P_{%
\mathcal{N}\left( T\right) ^{\bot }}.
\end{equation*}%
Thus, (\ref{eqn3.0_2}) holds.
\end{proof}

\begin{lemma}
\label{lem3.1_0} Let $T\in \mathcal{B}\left( X,Y\right) $ have LPA $%
\{(X_{n},T_{n})\}_{n\in
\mathbb{N}
}$, $\left\{ Q_{n}\right\} _{n\in
\mathbb{N}
}$ be defined as (\ref{eqn3.0}). Then, for all $n\in
\mathbb{N}
$,%
\begin{equation*}
1\leq \left\Vert I-Q_{n}\right\Vert <+\infty ,
\end{equation*}%
\begin{equation*}
\left\Vert P_{T^{\ast }T\left( X_{n}\right) }-P_{T^{\dag }T\left(
X_{n}\right) }\right\Vert =\sqrt{1-{\Vert I-Q_{n}\Vert }^{-2}}\in \left[
0,1\right) .
\end{equation*}
\end{lemma}

\begin{proof}
Since $\dim \mathcal{R}\left( Q_{n}\right) <\dim X$ ($=\infty $), then $%
I-Q_{n}\neq 0$, and therefore%
\begin{equation*}
1\leq \left\Vert I-Q_{n}\right\Vert <+\infty \quad (\forall n\in
\mathbb{N}
).
\end{equation*}%
By Lemma \ref{lem3.0}, we have
\begin{eqnarray*}
\left\Vert P_{\mathcal{R}(I-Q_{n})}-P_{\mathcal{R}(I-Q_{n}^{\ast
})}\right\Vert &=&\left\Vert P_{\mathcal{R}(I-Q_{n})}P_{\mathcal{N}%
(I-Q_{n})}\right\Vert \\
&=&\left\Vert P_{\mathcal{N}(I-Q_{n})}P_{\mathcal{R}(I-Q_{n})}\right\Vert \\
&=&\sqrt{1-{\Vert I-Q_{n}\Vert }^{-2}}\in \left[ 0,1\right) .
\end{eqnarray*}%
Note that, by Lemma \ref{lem3.1},
\begin{equation*}
\mathcal{N}\left( Q_{n}\right) =\left[ T^{\ast }T\left( X_{n}\right) \right]
^{\bot },\quad \mathcal{R}\left( Q_{n}\right) =T^{\dag }T\left( X_{n}\right)
.
\end{equation*}%
Hence, there holds%
\begin{eqnarray*}
\left\Vert P_{T^{\ast }T\left( X_{n}\right) }-P_{T^{\dag }T\left(
X_{n}\right) }\right\Vert &=&\left\Vert P_{\mathcal{N}(Q_{n})}-P_{\mathcal{R}%
(Q_{n})^{\bot }}\right\Vert \\
&=&\left\Vert P_{\mathcal{R}(I-Q_{n})}-P_{\mathcal{R}(I-Q_{n}^{\ast
})}\right\Vert \\
&=&\sqrt{1-{\Vert I-Q_{n}\Vert }^{-2}}\in \left[ 0,1\right) .
\end{eqnarray*}
\end{proof}

Now, we can present the proof of Theorem \ref{the1.1}.

\begin{proof}[Proof of Theorem \protect\ref{the1.1}]
By (\ref{eqn1.5}) we see that
\begin{equation*}
\sin \theta _{n}=\limfunc{gap}\left( T^{\dag }T\left( X_{n}\right) ,T^{\ast
}T\left( X_{n}\right) \right) ,\quad n\in
\mathbb{N}
.
\end{equation*}%
This, by Lemma \ref{lem3.0_2} and Lemma \ref{lem3.1_0}, gives that%
\begin{equation}
\sin \theta _{n}=\left\Vert P_{T^{\dag }T\left( X_{n}\right) }-P_{T^{\ast
}T\left( X_{n}\right) }\right\Vert =\sqrt{1-{\Vert I-Q_{n}\Vert }^{-2}}\in %
\left[ 0,1\right) ,\quad n\in
\mathbb{N}
.  \label{eqn3.7}
\end{equation}

\textrm{(a)} \ It is clear from (\ref{eqn3.7}) that $\left\{ \theta
_{n}\right\} $ is a sequence in the interval $\left[ 0,\frac{\pi }{2}\right)
$, and%
\begin{equation}
{\Vert I-Q_{n}\Vert }=\frac{1}{\cos \theta _{n}}=\sqrt{1+\tan ^{2}\theta _{n}%
}\in \left[ 1,+\infty \right) ,\quad n\in
\mathbb{N}
.  \label{eqn3.1}
\end{equation}
Note that%
\begin{equation*}
P_{\mathcal{N}\left( T\right) ^{\bot }}T_{n}^{\dag }T=T^{\dag }TT_{n}^{\dag
}T=T^{\dag }T_{n}T_{n}^{\dag }T=T^{\dag }P_{\mathcal{R}\left( T_{n}\right)
}T=Q_{n},
\end{equation*}%
and therefore%
\begin{equation}
Q_{n}=P_{\mathcal{N}\left( T\right) ^{\bot }}T_{n}^{\dag }T\quad \forall
n\in
\mathbb{N}
.  \label{eqn3.1_0}
\end{equation}

Assume that (\ref{eqn1.3}) is valid, this with (\ref{eqn3.1_0}) implies that%
\begin{equation}
\sup_{n}\left\Vert Q_{n}\right\Vert <+\infty ,\quad \sup_{n}\left\Vert
I-Q_{n}\right\Vert <+\infty .  \label{eqn3.2}
\end{equation}%
Hence, from (\ref{eqn3.2}) and (\ref{eqn3.1}) we obtain%
\begin{equation*}
\sup_{n}\theta _{n}<\frac{\pi }{2}.
\end{equation*}%
In addition, by the assertion (b) of Lemma \ref{lem2.0}, we also have%
\begin{equation*}
\text{(\ref{eqn1.3}) }\Longrightarrow \text{ (\ref{eqn1.4}).}
\end{equation*}%
Thus, there holds%
\begin{equation*}
\text{(\ref{eqn1.3})}\Longrightarrow \left\{
\begin{array}{l}
\text{(\ref{eqn1.4}),} \\
\sup_{n}\theta _{n}<\frac{\pi }{2}.%
\end{array}%
\right.
\end{equation*}%
Inversely, suppose (\ref{eqn1.4}) and $\sup_{n}\theta _{n}<\frac{\pi }{2}$
hold. Then (\ref{eqn3.2}) holds by (\ref{eqn3.1}). Due to the assertion (c)
of Lemma \ref{lem2.0}, it follows from (\ref{eqn1.4}) with $\dim \mathcal{N}%
\left( T\right) <\infty $ that there is a $n_{\ast }\in
\mathbb{N}
$ such that (\ref{eqn2.1}) holds and hence%
\begin{equation}
{\mathcal{R}(T_{n}^{\dag })}={\mathcal{N}(T_{n})}^{\perp }={{\mathcal{N}(T)}}%
^{\perp }\cap X_{n}\text{ for }n\geq n_{\ast }\text{.}  \label{eqn3.3}
\end{equation}%
From (\ref{eqn3.1_0}) and (\ref{eqn3.3}) we have%
\begin{equation}
T_{n}^{\dag }T=Q_{n}+P_{\mathcal{N}\left( T\right) }T_{n}^{\dag
}T=Q_{n}\quad \text{for }n\geq n_{\ast }.  \label{eqn3.4}
\end{equation}%
Hence, (\ref{eqn1.3}) follows from (\ref{eqn3.4}) and (\ref{eqn3.2}).

\textrm{(b)} \ Let (\ref{eqn1.4}) (with $\dim \mathcal{N}\left( T\right)
<\infty $) hold. Then (\ref{eqn2.0}), (\ref{eqn2.1}), (\ref{eqn3.1}), and (%
\ref{eqn3.4}) hold by Lemmas \ref{lem2.0}, \ref{lem3.0_2}, and \ref{lem3.1_0}%
. Hence, there is a $n_{\ast }\in
\mathbb{N}
$ such that
\begin{equation}
T_{n}^{\dag }-T^{\dag }=\left( Q_{n}-I\right) P_{X_{n}^{\,\bot }}T^{\dag
}\quad \text{on }\mathcal{D}\left( T^{\dag }\right) ,\quad \text{for }n\geq
n_{\ast },  \label{eqn3.5}
\end{equation}%
and%
\begin{equation}
\left.
\begin{array}{c}
X_{n}=\mathcal{N}\left( T\right) \overset{\bot }{\oplus }\left( X_{n}\cap
\mathcal{N}\left( T\right) ^{\bot }\right) , \\
P_{X_{n}}=P_{\mathcal{N}\left( T\right) }+P_{X_{n}\cap \mathcal{N}\left(
T\right) ^{\bot }},%
\end{array}%
\right\} \quad \text{ for }n\geq n_{\ast }\text{.}  \label{eqn3.6}
\end{equation}%
From (\ref{eqn3.5}) and (\ref{eqn3.1}) it is clear that, for $n\geq n_{\ast
} $,
\begin{equation}
\left\Vert T_{n}^{\dag }y-T^{\dag }y\right\Vert \leq \sqrt{1+\tan ^{2}\theta
_{n}}\mathrm{dist\hspace{0.15em}}\left( T^{\dag }y,X_{n}\right) ,\quad y\in
\mathcal{D}\left( T^{\dag }\right) .  \label{eqn3.1_1}
\end{equation}%
Next, we need only to show that, for $n\geq n_{\ast }$,
\begin{equation}
\theta _{n}=0\Longleftrightarrow T_{n}^{\dag }=P_{X_{n}}T^{\dag }\text{ on }%
\mathcal{D}\left( T^{\dag }\right) \Longleftrightarrow \mathcal{N}\left(
T\right) +T^{\ast }T\left( X_{n}\right) \subseteq X_{n}.  \label{eqn3.1_2}
\end{equation}

If $\theta _{n}=0$ ($n\geq n_{\ast }$), then from (\ref{eqn3.1_1}) and ${%
\mathcal{R}(T_{n}^{\dag })}={\mathcal{N}(T_{n})}^{\perp }\subseteq X_{n}$ it
is clear that%
\begin{equation*}
T_{n}^{\dag }=P_{X_{n}}T^{\dag }\text{ on }\mathcal{D}\left( T^{\dag
}\right) ,\quad n\geq n_{\ast }\text{.}
\end{equation*}

If the above equality is valid, it follows from (\ref{eqn3.4}) and (\ref%
{eqn3.6}) that%
\begin{equation*}
Q_{n}=T_{n}^{\dag }T=P_{X_{n}}P_{\mathcal{N}\left( T\right) ^{\bot
}}=P_{X_{n}\cap \mathcal{N}\left( T\right) ^{\bot }},\quad n\geq n_{\ast }%
\text{,}
\end{equation*}%
and therefore $Q_{n}$ is an orthogonal-projection. By use of (\ref{eqn3.0_1}%
) (of Lemma \ref{lem3.1}), there hold%
\begin{eqnarray*}
T^{\ast }T\left( X_{n}\right) &=&\mathcal{N}\left( Q_{n}\right) ^{\bot }=%
\mathcal{R}\left( Q_{n}\right) =T^{\dag }T\left( X_{n}\right) \\
&\subseteq &T^{-1}T\left( X_{n}\right) =X_{n}+\mathcal{N}\left( T\right)
=X_{n},\quad n\geq n_{\ast }\text{,}
\end{eqnarray*}%
this gives%
\begin{equation*}
T^{\ast }T\left( X_{n}\right) +\mathcal{N}\left( T\right) \subseteq
X_{n},\quad n\geq n_{\ast }\text{.}
\end{equation*}

If the above inclusions hold, then%
\begin{equation*}
T^{\ast }T\left( X_{n}\right) \subseteq T^{\dag }T\left( X_{n}\right) ,\quad
n\geq n_{\ast },
\end{equation*}%
and therefore%
\begin{equation*}
T^{\dag }T\left( X_{n}\right) =T^{\ast }T\left( X_{n}\right) ,\quad n\geq
n_{\ast }.
\end{equation*}%
This gives that%
\begin{equation*}
\sin \theta _{n}=\limfunc{gap}\left( T^{\dag }T\left( X_{n}\right) ,T^{\ast
}T\left( X_{n}\right) \right) =0,\quad n\geq n_{\ast },
\end{equation*}%
that is, $\theta _{n}=0$ ($n\geq n_{\ast }$). Now, (\ref{eqn3.1_2}) is
proved.
\end{proof}

\begin{remark}
\label{Remark3.0} From the proof of Theorem \ref{the1.1}, we obtain that,
even if $\dim \mathcal{N}\left( T\right) =\infty $,
\begin{equation*}
\text{(\ref{eqn1.3})}\Longrightarrow \left\{
\begin{array}{l}
\text{(\ref{eqn1.4}),} \\
\sup_{n}\theta _{n}<\frac{\pi }{2}.%
\end{array}%
\right.
\end{equation*}
\end{remark}

\section{Examples and Remarks}

In this section, the two unconvergence examples of Seidman and Du will be
restudied under the concepts of offset angle and kernel approximability.
This can further explain the relations among the three concepts of strong
convergence,\ offset angle\emph{\ }and kernel approximability, and also
leads to some remarks on Theorems \ref{the1.0} and \ref{the1.1}.

\begin{example}[{Seidman's Example \protect\cite[Example 3.1]%
{seidman1980nonconvergence}}]
Let
\begin{equation*}
X:=l^{2},\quad X_{n}:=\mathrm{span\hspace{0.15em}}\left\{ \mathbf{e}^{1},%
\mathbf{e}^{2},\ldots ,\mathbf{e}^{n}\right\} ,
\end{equation*}%
where%
\begin{equation*}
\mathbf{e}^{k}=(0,\ldots ,0,\underset{\overset{\uparrow }{k\mathrm{th\hspace{%
0.15em}}}}{1},0,\ldots )\quad \forall \ k\in
\mathbb{N}
,
\end{equation*}%
and $T:X\rightarrow X$ be given in the form
\begin{equation*}
T\mathbf{x}:=\sum\limits_{k=1}^{\infty }\left( \alpha _{k}\xi _{k}+\beta
_{k}\xi _{1}\right) \mathbf{e}^{k}\quad \forall \mathbf{x}%
=\sum\limits_{k=1}^{\infty }\xi _{k}\mathbf{e}^{k}\in X,
\end{equation*}%
with
\begin{equation*}
\alpha _{j}:=\left\{
\begin{array}{rl}
j^{-1}, & \quad j\text{ odd,} \\
j^{-3}, & \quad j\text{ even,}%
\end{array}%
\right. \qquad \beta _{j}:=\left\{
\begin{array}{rl}
0, & \quad j=1, \\
j^{-1}, & \quad j>1\text{.}%
\end{array}%
\right.
\end{equation*}%
This defines a compact, injective linear operator $T$ with dense range, its
LPA\textbf{\ }$\{(X_{n},T_{n})\}_{n\in
\mathbb{N}
}$ has kernel approximability ((\ref{eqn1.4}) holds), and the offset angle
sequence $\left\{ \theta _{n}\right\} $ satisfies%
\begin{equation}
0<\theta _{n}=\arcsin \limfunc{gap}(X_{n},T^{\ast }T\left( X_{n}\right) )<%
\frac{\pi }{2}\ \text{(}\forall \,n\text{),}  \label{eqn4.0}
\end{equation}%
and%
\begin{equation}
\sup_{n}\theta _{n}=\frac{\pi }{2}.  \label{eqn4.1}
\end{equation}
\end{example}

\begin{remark}
For the LPA $\{(X_{n},T_{n})\}_{n\in
\mathbb{N}
}$ of Seidman's example, the weak/strong convergence or the condition (\ref%
{eqn1.3}) do not hold. The reason for such a unconvergence is caused by the
offset angle $\theta _{n}$\ tending towards perpendicular angle as (\ref%
{eqn4.1}) shows. \textbf{Note}: LPA $\{(X_{n},T_{n})\}_{n\in
\mathbb{N}
}$\ has kernel-approximability since $\mathcal{N}(T)=\left\{ 0\right\} $,
and therefore has inverse-graph approximability and bounded-weak convergence%
\textbf{\ }by Theorem \ref{the1.0}.
\end{remark}

\begin{proof}
It is easy to see that
\begin{equation*}
T^{\ast }\mathbf{y}=\sum\limits_{k=1}^{\infty }\eta _{k}(\overline{\alpha
_{k}}\mathbf{e}^{k}+\overline{\beta _{k}}\mathbf{e}^{1})\quad \forall
\mathbf{y}=\sum\limits_{k=1}^{\infty }\eta _{k}\mathbf{e}^{k}\in X,
\end{equation*}%
and
\begin{equation*}
T^{\ast }T\mathbf{x}=\sum\limits_{k=1}^{\infty }(\alpha _{k}\xi _{k}+\beta
_{k}\xi _{1})(\overline{\alpha _{k}}\mathbf{e}^{k}+\overline{\beta _{k}}%
\mathbf{e}^{1})\quad \forall \mathbf{x}=\sum\limits_{k=1}^{\infty }\xi _{k}%
\mathbf{e}^{k}\in X.
\end{equation*}%
It is clear that
\begin{equation*}
\mathcal{N}(T)=\left\{ 0\right\} \subseteq X_{n}\ \text{(}\forall \,n\text{%
),\qquad }T^{\dag }=T^{-1}:\mathcal{R}\left( T\right) \subseteq X\rightarrow
X,
\end{equation*}%
\begin{equation*}
\frac{\pi ^{2}}{6}\mathbf{e}^{1}+\sum\limits_{k=2}^{\infty }\alpha _{k}\beta
_{k}\mathbf{e}^{k}=T^{\ast }T\mathbf{e}^{1}\in T^{\ast }T\left( X_{n}\right)
\setminus X_{n}\ \text{(}\forall \,n\text{).}
\end{equation*}%
This implies (\ref{eqn1.4}) and (\ref{eqn4.0}).\ Next, we will show (\ref%
{eqn4.1}) holds.

We rewrite $T^{\ast }T\mathbf{x}$ for $\mathbf{x=x}_{n}:\mathbf{=}%
\sum\limits_{k=1}^{n}\xi _{k}\mathbf{e}^{k}\in X_{n}$ ($\xi _{k}:=0\ \forall
k\geq n+1$),%
\begin{eqnarray*}
T^{\ast }T\mathbf{x}_{n} &=&\left[ \sum\limits_{k=2}^{\infty }(\alpha
_{k}\xi _{k}+\beta _{k}\xi _{1})\beta _{k}+\alpha _{1}^{2}\xi _{1}\right]
\mathbf{e}^{1} \\
&&+\sum\limits_{k=2}^{n}(\alpha _{k}\xi _{k}+\beta _{k}\xi _{1})\alpha _{k}%
\mathbf{e}^{k}+\sum\limits_{k=n+1}^{\infty }\alpha _{k}\beta _{k}\xi _{1}%
\mathbf{e}^{k},
\end{eqnarray*}%
then we have
\begin{eqnarray*}
P_{X_{n}}(T^{\ast }T\mathbf{x}_{n}) &=&\left[ \sum\limits_{k=2}^{\infty
}(\alpha _{k}\xi _{k}+\beta _{k}\xi _{1})\beta _{k}+\alpha _{1}^{2}\xi _{1}%
\right] \mathbf{e}^{1} \\
&&+\sum\limits_{k=2}^{n}(\alpha _{k}\xi _{k}+\beta _{k}\xi _{1})\alpha _{k}%
\mathbf{e}^{k},
\end{eqnarray*}%
and
\begin{equation*}
P_{X_{n}^{\bot }}(T^{\ast }T\mathbf{x}_{n})=\sum\limits_{k=n+1}^{\infty
}\alpha _{k}\beta _{k}\xi _{1}\mathbf{e}^{k}.
\end{equation*}%
Now we take $\mathbf{x}_{n}$ that satisfies
\begin{equation*}
\xi _{1}=1,\ \xi _{k}=-\frac{\beta _{k}\xi _{1}}{\alpha _{k}},\quad
k=2,...,n-1,
\end{equation*}%
and let $\lambda :=\alpha _{n}\xi _{n}+\beta _{n}\xi _{1}$ be undetermined,
we have%
\begin{eqnarray*}
T^{\ast }T\mathbf{x}_{n} &=&\left[ \lambda \beta _{n}+\alpha
_{1}^{2}+\sum\limits_{k=n+1}^{\infty }\beta _{k}^{2}\right] \mathbf{e}%
^{1}+\lambda \alpha _{n}\mathbf{e}^{n}+\sum\limits_{k=n+1}^{\infty }\alpha
_{k}\beta _{k}\mathbf{e}^{k} \\
&=&\left( \lambda \beta _{n}+c_{n}^{2}\right) \mathbf{e}^{1}+\lambda \alpha
_{n}\mathbf{e}^{n}+\sum\limits_{k=n+1}^{\infty }\alpha _{k}\beta _{k}\mathbf{%
e}^{k},
\end{eqnarray*}%
where%
\begin{equation*}
c_{n}=\alpha _{1}^{2}+\sum\limits_{k=n+1}^{\infty }\beta _{k}^{2}\leq
\sum\limits_{k=1}^{\infty }\frac{1}{k^{2}}=\frac{\pi ^{2}}{6}.
\end{equation*}%
Take $\lambda =\frac{-c_{n}\beta _{n}}{\alpha _{n}^{2}+\beta _{n}^{2}}$,
then
\begin{equation*}
\left\{
\begin{array}{l}
\Vert T^{\ast }T\mathbf{x}_{n}\Vert ^{2}=\frac{c_{n}^{2}\alpha _{n}^{2}}{%
\alpha _{n}^{2}+\beta _{n}^{2}}+\sum\limits_{k=n+1}^{\infty }(\alpha
_{k}\beta _{k})^{2}\leq \frac{\frac{\pi ^{4}}{36}\alpha _{n}^{2}}{\alpha
_{n}^{2}+\beta _{n}^{2}}+\sum\limits_{k=n+1}^{\infty }(\alpha _{k}\beta
_{k})^{2}, \\
\left\Vert P_{X_{n}^{\bot }}(T^{\ast }T\mathbf{x}_{n})\right\Vert
^{2}=\sum\limits_{k=n+1}^{\infty }\left( \alpha _{k}\beta _{k}\right) ^{2},%
\end{array}%
\right.
\end{equation*}%
Note that, for $n$ even,
\begin{equation*}
\frac{\frac{\pi ^{4}}{36}\alpha _{n}^{2}}{\alpha _{n}^{2}+\beta _{n}^{2}}=%
\frac{O(n^{-6})}{O(n^{-6})+O(n^{-2})}=O(n^{-4}),\quad
\sum\limits_{k=n+1}^{\infty }(\alpha _{k}\beta _{k})^{2}=O(n^{-3}).
\end{equation*}%
According to the definition of $\delta (T^{\ast }T\left( X_{n}\right)
,X_{n}) $, we have
\begin{equation*}
\delta (T^{\ast }T\left( X_{n}\right) ,X_{n})=\sup_{\mathbf{x}\in X_{n}}%
\frac{\Vert P_{X_{n}^{\bot }}(T^{\ast }T\mathbf{x})\Vert }{\Vert T^{\ast }T%
\mathbf{x}\Vert }.
\end{equation*}%
So for $n$ even, we obtain that
\begin{eqnarray*}
\left[ \limfunc{gap}(X_{n},T^{\ast }T\left( X_{n}\right) )\right] ^{2} &=&%
\left[ \delta (T^{\ast }TX_{n},X_{n})\right] ^{2}\qquad \text{(by Lemma \ref%
{lem3.0_2})} \\
&\geq &\frac{1}{1+O(n^{-1})}.
\end{eqnarray*}%
This gives (\ref{eqn4.1}).
\end{proof}

\begin{example}[Best-LPA]
Let $K:X\rightarrow Y$ be a compact linear operator, $\left\{ \left( \sigma
_{n};v_{n},u_{n}\right) \right\} _{n=0}^{\infty }$ be the singular system
for $K$. Then the $\left\{ \sigma _{n}^{2}\right\} $ are the nonzero
eigenvalues of the self-adjoint operator $K^{\ast }K$ (and also of $KK^{\ast
}$), written down in decreasing order with multiplicity, $\sigma _{n}>0$,
the $\left\{ v_{n}\right\} _{n=0}^{\infty }$\ is a corresponding complete
orthonormal system of eigenvectors of $K^{\ast }K$ (which spans $\overline{%
\mathcal{R}(K^{\ast })}=\overline{\mathcal{R}(K^{\ast }K)}$), and the $%
\left\{ u_{n}\right\} _{n=0}^{\infty }$ is defined by vectors%
\begin{equation*}
u_{n}:=\frac{Kv_{n}}{\left\Vert Kv_{n}\right\Vert }\quad \left( n=0,1,\cdots
\right) .
\end{equation*}%
The $\left\{ u_{n}\right\} _{n=0}^{\infty }$\ is a complete orthonormal
system of eigenvectors of $KK^{\ast }$ and span $\overline{\mathcal{R}(K)}=%
\overline{\mathcal{R}(KK^{\ast })}$, and the following formulae hold:%
\begin{equation*}
Kv_{n}=\sigma _{n}u_{n},\quad K^{\ast }u_{n}=\sigma _{n}v_{n}.
\end{equation*}%
Now, if $\dim \mathcal{N}\left( K\right) <\infty $, and if $\left\{
m_{n}\right\} \subseteq
\mathbb{N}
$ is an increasing sequence with $\lim_{n\rightarrow \infty }m_{n}=\infty $,
take%
\begin{equation*}
X_{n}:=\mathcal{N}\left( K\right) +\mathrm{span\hspace{0.15em}}\left\{
v_{1},v_{2},\cdots ,v_{m_{n}}\right\} \quad \forall n\in \mathbb{N}.
\end{equation*}%
Then%
\begin{equation*}
\mathcal{N}\left( K\right) +K^{\ast }K\left( X_{n}\right) =\mathcal{N}\left(
K\right) +\mathrm{span\hspace{0.15em}}\left\{ \sigma _{1}^{2}v_{1},\sigma
_{2}^{2}v_{2},\cdots ,\sigma _{m_{n}}^{2}v_{m_{n}},\right\} \subseteq X_{n}.
\end{equation*}%
In this case, we have the best LPA $\{(X_{n},K_{n})\}_{n\in
\mathbb{N}
}$ for $K$:
\begin{equation*}
\mathcal{N}\left( K\right) \cap X_{n}=\mathcal{N}\left( K\right) ,\quad
\theta _{n}:=\arcsin \limfunc{gap}(K^{\dag }K\left( X_{n}\right) ,K^{\ast
}K\left( X_{n}\right) )=0\ \text{(}\forall \,n\text{),}
\end{equation*}%
and
\begin{equation*}
K_{n}^{\,\dag }=P_{X_{n}}K^{\dag }\text{ on }\mathcal{D}\left( K^{\dag
}\right) .
\end{equation*}
\end{example}

\begin{example}[{Du's Example \protect\cite[Example 2.10]{du2008finite}}]
Let
\begin{equation*}
X:=l^{2},\quad X_{n}:=\mathrm{span\hspace{0.15em}}\left\{ \mathbf{e}^{1},%
\mathbf{e}^{2},\ldots ,\mathbf{e}^{n}\right\} ,
\end{equation*}%
and $T:X\rightarrow X$ defined as follows:%
\begin{equation*}
Tx:=x-\left\langle x,\mathbf{e}\right\rangle \mathbf{e}\quad \text{ }\left(
x\in X\right) ,
\end{equation*}%
where $\left\langle \cdot ,\cdot \right\rangle $ is the inner product on $X$
and%
\begin{equation*}
\mathbf{e}:=\frac{\sqrt{3}}{2}\mathbf{e}^{1}+\frac{\sqrt{3}}{2^{2}}\mathbf{e}%
^{2}+\cdots +\frac{\sqrt{3}}{2^{k}}\mathbf{e}^{k}+\cdots .
\end{equation*}%
This operator satisfies%
\begin{equation*}
T=T^{2}=T^{\ast }=T^{\dag }\quad \text{and\quad }\mathcal{N}\left( T\right) =%
\mathrm{span\hspace{0.15em}}\left\{ \mathbf{e}\right\} ;
\end{equation*}%
that is, $T$ is the orthogonal projection of $X$ onto $\mathbf{e}^{\bot }$.
For its LPA $\{(X_{n},T_{n})\}_{n\in
\mathbb{N}
}$, where%
\begin{equation*}
T_{n}:=TP_{X_{n}},\quad P_{n}:=P_{X_{n}},\quad n\in
\mathbb{N}
,
\end{equation*}%
we have the following facts:\newline
\textrm{(a)} \ The offset angle $\theta _{n}$ satisfies%
\begin{equation*}
\theta _{n}:=\arcsin \limfunc{gap}(T^{\dag }T\left( X_{n}\right) ,T^{\ast
}T\left( X_{n}\right) )=0\ \text{(}\forall \,n\text{).}
\end{equation*}%
\textrm{(b)} \ The kernel approximability is invalid since%
\begin{equation*}
\mathcal{N}\left( T\right) \cap X_{n}=\left\{ 0\right\} \neq \mathcal{N}%
\left( T\right) \ \text{(}\forall \,n\text{).}
\end{equation*}%
\textrm{(c)} \ $T_{n}^{\dag }y=P_{n}y-4^{n}\left\langle \left(
I-P_{n}\right) y,e\right\rangle P_{n}e\ \left( \forall \ y\in X\right) $,
and for any fixed $y\in X,$%
\begin{equation*}
\left\{
\begin{array}{l}
\underset{n}{\sup }\left\Vert T_{n}^{\dag }y\right\Vert <\infty
\Longleftrightarrow \,\underset{n}{\sup }4^{n}\left\vert \left\langle \left(
I-P_{n}\right) y,e\right\rangle \right\vert <\infty \,,\medskip \\
\underset{n\rightarrow \infty }{\mathrm{w\hspace{0.15em}}\text{-}\lim }%
T_{n}^{\dag }y=T^{\dag }y\Longleftrightarrow \,\underset{n\rightarrow \infty
}{\lim }4^{n}\left\langle \left( I-P_{n}\right) y,e\right\rangle
=\left\langle y,e\right\rangle ;%
\end{array}%
\right.
\end{equation*}%
taking
\begin{equation*}
y:=\Bigg(\frac{\sqrt{3}}{4},\frac{3\sqrt{3}}{4^{2}},\ldots ,\frac{\left(
2^{n}-1\right) \sqrt{3}}{4^{n}},\ldots \Bigg),
\end{equation*}%
there hold%
\begin{equation*}
\underset{n}{\sup }\left\Vert T_{n}^{\dag }y\right\Vert <\infty \quad \text{%
and\quad }\underset{n\rightarrow \infty }{\mathrm{w\hspace{0.15em}}\text{-}%
\lim }T_{n}^{\dag }y\neq T^{\dag }y.
\end{equation*}%
That indicates that the bounded-weak convergence is invalid.\newline
\textrm{(d)} \ For any fixed $y\in \mathcal{D}\left( T^{\dag }\right) =X,$%
\begin{equation*}
\underset{n\rightarrow \infty }{\mathrm{w\hspace{0.15em}}\text{-}\lim }%
T_{n}^{\dag }y=T^{\dag }y\Longleftrightarrow \,\underset{n\rightarrow \infty
}{\mathrm{s\hspace{0.15em}}\text{-}\lim }T_{n}^{\dag }y=T^{\dag
}y\Longleftrightarrow \,\underset{n\rightarrow \infty }{\overline{\lim }}%
\left\Vert T_{n}^{\dag }y\right\Vert \leq \left\Vert T^{\dag }y\right\Vert
\,.
\end{equation*}
\end{example}

\begin{remark}
For the LPA $\{(X_{n},T_{n})\}_{n\in
\mathbb{N}
}$ of Du's example, the weak/strong convergence or the condition (\ref%
{eqn1.3}) do not hold. The reason for unconvergence is caused by the fact
that kernel approximability is invalid.
\end{remark}

\begin{proof}
(a) \ Since $T^{\dag }=T=T^{\ast }$, it is clear that $\theta _{n}=0\ $($%
\forall \,n$).

(b) \ See \cite[Proposition A.1]{du2008finite}.

(c) \ See \cite[Propositions A.2 and A.3]{du2008finite}.

(d) \ For any fixed $y\in X$,\ if $\mathrm{w\hspace{0.15em}}$-$%
\lim_{n\rightarrow \infty }T_{n}^{\dag }y=T^{\dag }y$, by (c) and $T^{\dag
}=T$,%
\begin{equation*}
\lim_{n\rightarrow \infty }4^{n}\left\langle \left( I-P_{n}\right)
y,e\right\rangle =\left\langle y,e\right\rangle ,
\end{equation*}
and therefore%
\begin{eqnarray*}
\left\Vert T_{n}^{\dag }y-T^{\dag }y\right\Vert &=&\left\Vert
P_{n}y-4^{n}\left\langle \left( I-P_{n}\right) y,\mathbf{e}\right\rangle
P_{n}\mathbf{e}-y+\left\langle y,\mathbf{e}\right\rangle \mathbf{e}%
\right\Vert \\
&\leq &\left\Vert P_{n}y-y\right\Vert +\left\vert 4^{n}\left\langle \left(
I-P_{n}\right) y,\mathbf{e}\right\rangle -\left\langle y,\mathbf{e}%
\right\rangle \right\vert +\left\vert \left\langle y,\mathbf{e}\right\rangle
\right\vert \left\Vert P_{n}\mathbf{e}-\mathbf{e}\right\Vert \\
&\rightarrow &0\quad \left( n\rightarrow \infty \right) ,
\end{eqnarray*}%
that is, $\mathrm{s\hspace{0.15em}}$-$\lim_{n\rightarrow \infty }T_{n}^{\dag
}y=T^{\dag }y$.

If $\mathrm{s\hspace{0.15em}}$-$\lim_{n\rightarrow \infty }T_{n}^{\dag
}y=T^{\dag }y$ holds, then $\overline{\lim }_{n\rightarrow \infty
}\left\Vert T_{n}^{\dag }y\right\Vert \leq \left\Vert T^{\dag }y\right\Vert $%
.

If $y\in \mathcal{D}\left( T^{\dag }\right) $ and $\overline{\lim }%
_{n\rightarrow \infty }\left\Vert T_{n}^{\dag }y\right\Vert \leq \left\Vert
T^{\dag }y\right\Vert $, then any subsequence $\left\{ T_{n_{k}}^{\dag
}y\right\} $ of $\left\{ T_{n}^{\dag }y\right\} $ has a subsequence, again
denoted by $\left\{ T_{n_{k}}^{\dag }y\right\} $, converging weakly to some $%
u\in X$, and%
\begin{equation}
\left\Vert u\right\Vert \leq \underset{k\rightarrow \infty }{\underline{\lim
}}\left\Vert T_{n_{k}}^{\dag }y\right\Vert \leq \underset{k\rightarrow
\infty }{\overline{\lim }}\left\Vert T_{n_{k}}^{\dag }y\right\Vert \leq
\underset{n\rightarrow \infty }{\overline{\lim }}\left\Vert T_{n}^{\dag
}y\right\Vert \leq \left\Vert T^{\dag }y\right\Vert .  \label{eqn4.2}
\end{equation}%
Since $T$ is bounded and%
\begin{equation*}
\left\{
\begin{array}{l}
\mathcal{N}\left( T_{n}\right) ^{\bot }=\left( \mathcal{N}\left( T\right)
\cap X_{n}\right) ^{\bot }\cap X_{n}\subseteq X_{n}, \\
\underset{n\rightarrow \infty }{\mathrm{s\hspace{0.15em}}\!\text{%
-\negthinspace }\lim }P_{\mathcal{R}\left( T_{n}\right) }=P_{\overline{%
\mathcal{R}\left( T\right) }},%
\end{array}%
\right. \text{(by statement (a) of lemma \ref{lem2.0}) }
\end{equation*}%
we have that, for all $v\in X$,
\begin{eqnarray*}
\left\langle T\left( u-T^{\dag }y\right) ,v\right\rangle &=&\left\langle
Tu-TT_{n_{k}}^{\dag }y,v\right\rangle +\left\langle TT_{n_{k}}^{\dag
}y-TT^{\dag }y,v\right\rangle \\
&=&\left\langle u-T_{n_{k}}^{\dag }y,T^{\ast }v\right\rangle +\left\langle
T_{n_{k}}T_{n_{k}}^{\dag }y-TT^{\dag }y,v\right\rangle \\
&=&\left\langle u-T_{n_{k}}^{\dag }y,T^{\ast }v\right\rangle +\left\langle
P_{\mathcal{R}\left( T_{n_{k}}\right) }y-P_{\overline{\mathcal{R}\left(
T\right) }}y,v\right\rangle \\
&\rightarrow &0\quad \left( k\rightarrow \infty \right) \text{,}
\end{eqnarray*}%
and hence
\begin{equation*}
u\in T^{\dag }y+\mathcal{N}\left( T\right) ,\quad \left\Vert u\right\Vert
^{2}=\left\Vert T^{\dag }y\right\Vert ^{2}+\left\Vert u-T^{\dag
}y\right\Vert ^{2}.
\end{equation*}%
This with (\ref{eqn4.2}) implies that%
\begin{equation*}
\underset{k\rightarrow \infty }{\mathrm{w\hspace{0.15em}}\text{-}\lim }%
T_{n_{k}}^{\dag }y=u=T^{\dag }y\quad \text{with\quad }\lim_{k\rightarrow
\infty }\left\Vert T_{n_{k}}^{\dag }y\right\Vert =\left\Vert T^{\dag
}y\right\Vert ,
\end{equation*}%
that is,%
\begin{equation*}
\underset{k\rightarrow \infty }{\mathrm{s\hspace{0.15em}}\text{-}\lim }%
T_{n_{k}}^{\dag }y=T^{\dag }y.
\end{equation*}%
Thus, every subsequence of $\left\{ T_{n}^{\dag }y\right\} $ has a
subsequence converging weakly to $T^{\dag }y$ and hence%
\begin{equation*}
\underset{n\rightarrow \infty }{\mathrm{w\hspace{0.15em}}\text{-}\lim }%
T_{n}^{\dag }y=T^{\dag }y,\quad \text{in fact,\quad }\underset{n\rightarrow
\infty }{\mathrm{s\hspace{0.15em}}\text{-}\lim }T_{n}^{\dag }y=T^{\dag }y.
\end{equation*}%
Note that the proof of this paragraph applies to a general LPA!
\end{proof}

\begin{remark}
From the above proof we actually obtain a more general result compared with
the one obtained by Luecke and Hickey \cite[Theorem 11]{Luecke-Hickey1985},
who give the result in the situation of LPA $\{(X_{n},T_{n})\}$ with an
increasing sequence:
\begin{equation*}
X_{0}\subseteq X_{1}\subseteq X_{2}\subseteq \cdots \text{\quad with}\quad
\overline{\overset{\infty }{\underset{n=0}{\cup }}X_{n}}=X\,.
\end{equation*}%
The extended version is as below:
\end{remark}

\begin{proposition}
Let $T\in \mathcal{B}\left( X,Y\right) $ have LPA $\{(X_{n},T_{n})\}_{n\in
\mathbb{N}
}$. For a fixed $y\in \mathcal{D}\left( T^{\dag }\right) $ there holds%
\begin{equation*}
\underset{n\rightarrow \infty }{\mathrm{s\hspace{0.15em}}\text{-}\lim }%
T_{n}^{\dag }y=T^{\dag }y\,\Longleftrightarrow \,\underset{n\rightarrow
\infty }{\overline{\lim }}\left\Vert T_{n}^{\dag }y\right\Vert \leq
\left\Vert T^{\dag }y\right\Vert .\qquad \text{\emph{(True)}}
\end{equation*}
\end{proposition}

\begin{remark}
It should be noted that the following proposition is false.
\end{remark}

\begin{proposition}
Let $T\in \mathcal{B}\left( X,Y\right) $ have LPA $\{(X_{n},T_{n})\}_{n\in
\mathbb{N}
}$. For a fixed $y\in \mathcal{D}\left( T^{\dag }\right) $,%
\begin{equation*}
\underset{n\rightarrow \infty }{\mathrm{\limfunc{w}\hspace{0.15em}}\text{-}%
\lim }T_{n}^{\dag }y=T^{\dag }y\,\Longleftrightarrow \,\sup_{n}\left\Vert
T_{n}^{\dag }y\right\Vert <+\infty .\qquad \text{\emph{(False)}}
\end{equation*}
\end{proposition}

\section{Conclusions}

In this paper, we propose two concepts -- the offset angle and the kernel
approximability of LPA, and show their roles in the convergence of
least-squares projection method. These concepts come from the wish to
understand two counter-examples (due to Seidman \cite%
{seidman1980nonconvergence} and Du \cite{du2008finite}) which respectively
represent two important cases in which least-squares projection method fails
to converge. Let us reformulate the concepts of the offset angles, the
kernel approximability, and the strong/weak convergence of least-squares
projection method:

\begin{itemize}
\item $\emph{Offset}$ $\emph{Angle}$:%
\begin{equation*}
\theta _{n}:=\arcsin \limfunc{gap}\left( T^{\dag }T\left( X_{n}\right)
,T^{\ast }T\left( X_{n}\right) \right) \quad \forall n\in
\mathbb{N}
.
\end{equation*}

\item \emph{Kernel Approximability}:
\begin{equation*}
\mathcal{N}\left( T\right) =\left\{ x\in X\,:\lim_{n\rightarrow \infty }%
\mathrm{dist\hspace{0.15em}}\left( x,\mathcal{N}\left( T\right) \cap
X_{n}\right) =0\right\} .
\end{equation*}

\item \emph{Strong Convergence (}$\underset{n\rightarrow \infty }{\mathrm{s%
\hspace{0.15em}}\text{-}\lim }T_{n}^{\,\dag }=T^{\dag }$ on $\mathcal{D}%
\left( T^{\dag }\right) $\emph{)}:%
\begin{equation*}
\lim_{n\rightarrow \infty }T_{n}^{\,\dag }y=T^{\dag }y\quad \forall \,y\in
\mathcal{D}\left( T^{\dag }\right) .
\end{equation*}

\item \emph{Weak Convergence (}$\underset{n\rightarrow \infty }{\mathrm{w%
\hspace{0.15em}}\text{-}\lim }T_{n}^{\,\dag }=T^{\dag }$ on $\mathcal{D}%
\left( T^{\dag }\right) $\emph{)}:
\begin{equation*}
\underset{n\rightarrow \infty }{\mathrm{w\hspace{0.15em}}\text{-}\lim }%
T_{n}^{\,\dag }y=T^{\dag }y\quad \forall \,y\in \mathcal{D}\left( T^{\dag
}\right) .
\end{equation*}
\end{itemize}

To understand the relevance among the four concepts, we make a table to show
the true and false status of these concepts in the several examples we have
mentioned (T for true, F for false).%
\begin{equation*}
\begin{tabular}{|c|c|c|c|}
\hline
Examples & $\sup_{n}\theta _{n}<\frac{\pi}{2}$ & Kernel Approximability & Strong/weak
Convergence \\ \hline
Seidman's & F & T & F \\ \hline
Du's & T & F & F \\ \hline
Best-LPA & T & T & T \\ \hline
\end{tabular}%
\end{equation*}

From this table we can see that the conditions \textquotedblleft kernel
approximability\textquotedblright\ and \textquotedblleft $\sup_{n}\theta
_{n}<\frac{\pi }{2}$\textquotedblright\ are two mutually independent important
factors of a convergent LPA $\{(X_{n},T_{n})\}$. In fact, both of them are
necessary if we want the convergence of LPA $\{(X_{n},T_{n})\}$, for we have
that (see Remak \ref{Remark3.0})
\begin{equation*}
\text{weak convergence}\Longleftrightarrow \text{strong convergence}%
\Longrightarrow \left\{
\begin{array}{l}
\text{kernel approximability,} \\
\sup_{n}\theta _{n}<\frac{\pi }{2}\text{.}%
\end{array}%
\right.
\end{equation*}%
We hope the three are equivalent, if so, the problem of strong/weak
convergence of least-squares projection method could be divided into two
subproblems, namely, the kernel approximability problem and the offset angle
problem.

In this paper, the main result (Theorem \ref{the1.1}) we get is: If $\dim
\mathcal{N}\left( T\right) <\infty $, then
\begin{equation*}
\left.
\begin{array}{l}
\text{kernel approximability} \\
\sup_{n}\theta _{n}<\frac{\pi }{2}%
\end{array}%
\right\} \Longleftrightarrow \text{strong convergence}\Longleftrightarrow
\text{weak convergence,}
\end{equation*}%
and when the kernel approximability is valid, for $n$ large enough,
\begin{equation*}
\left\Vert T_{n}^{\dag }y-T^{\dag }y\right\Vert \leq \sqrt{1+\tan ^{2}\theta
_{n}}\mathrm{\limfunc{dist}}\left( T^{\dag }y,X_{n}\right) \quad \forall
y\in \mathcal{D}\left( T^{\dag }\right) .
\end{equation*}%
Now, we consider the three special cases of \textquotedblleft $\mathcal{N}%
\left( T\right) =\left\{ 0\right\} $\textquotedblright , \textquotedblleft $%
\mathcal{R}\left( T\right) =\overline{\mathcal{R}\left( T\right) }$%
\textquotedblright ,\ and \textquotedblleft $\mathcal{N}\left( T\right)
=\left\{ 0\right\} \ \&\ \mathcal{R}\left( T\right) =\overline{\mathcal{R}%
\left( T\right) }$\textquotedblright , then there hold the following
interesting corollaries:

\begin{itemize}
\item If $\mathcal{N}\left( T\right) =\left\{ 0\right\} $, then, for any LPA
$\{(X_{n},T_{n})\}_{n\in
\mathbb{N}
}$ for $T$,
\begin{equation*}
\sup_{n}\theta _{n}<\frac{\pi }{2}\Longleftrightarrow \text{strong
convergence,}
\end{equation*}%
and for all $n\in
\mathbb{N}
$,
\begin{equation*}
\left\Vert T_{n}^{\dag }y-T^{\dag }y\right\Vert \leq \sqrt{1+\tan ^{2}\theta
_{n}}\,\mathrm{\limfunc{dist}}\left( T^{\dag }y,X_{n}\right) \quad \forall
y\in \mathcal{D}\left( T^{\dag }\right) .
\end{equation*}

\item If $\dim \mathcal{N}\left( T\right) <\infty $ and $\mathcal{R}\left(
T\right) $ is closed, then, for any LPA $\{(X_{n},T_{n})\}_{n\in
\mathbb{N}
}$ for $T$,
\begin{equation*}
\left\{
\begin{array}{l}
\text{kernel approximability}\Longleftrightarrow \text{strong convergence,}
\\[5pt]
\sup_{n}\theta _{n}<\frac{\pi }{2}\text{,}%
\end{array}%
\right.
\end{equation*}%
and when the kernel approximability holds, for $n$ large enough,
\begin{equation*}
\left\Vert T_{n}^{\dag }y-T^{\dag }y\right\Vert \leq \sqrt{1+\tan ^{2}\theta
_{n}}\,\mathrm{\limfunc{dist}}\left( T^{\dag }y,X_{n}\right) \quad \forall
y\in Y.
\end{equation*}

\item If $\mathcal{N}\left( T\right) =\left\{ 0\right\} $ and $\mathcal{R}%
\left( T\right) $ is closed, then for any LPA $\{(X_{n},T_{n})\}_{n\in
\mathbb{N}
}$ for $T$,
\begin{equation*}
\sup_{n}\theta _{n}<\frac{\pi }{2}\text{,}
\end{equation*}%
and for all $n\in
\mathbb{N}
$,
\begin{equation*}
\left\Vert T_{n}^{\dag }y-T^{\dag }y\right\Vert \leq \sqrt{1+\tan ^{2}\theta
_{n}}\,\mathrm{\limfunc{dist}}\left( T^{\dag }y,X_{n}\right) \quad \forall
y\in Y.
\end{equation*}%
Here, we remark that: If there are positive constants $\alpha $ and $\beta $
such that $T\in \mathcal{B}\left( X\right) $ satisfies%
\begin{equation*}
\left.
\begin{array}{l}
\left\vert \left\langle Tu,u\right\rangle \right\vert \geq \alpha \left\Vert
u\right\Vert ^{2}, \\
\left\vert \left\langle Tu,v\right\rangle \right\vert \leq \beta \left\Vert
u\right\Vert \left\Vert v\right\Vert%
\end{array}%
\right\} \text{ for all }u,v\in X,
\end{equation*}%
then $T$ has a bounded inverse $T^{-1}\in \mathcal{B}\left( X\right) $ by
Lax-Milgram theorem (see \cite[Theorem 13.26]{Kress1999}), and%
\begin{equation*}
\sup_{n}\sqrt{1+\tan ^{2}\theta _{n}}\leq \left\Vert T\right\Vert \left\Vert
T^{-1}\right\Vert \leq \frac{\beta }{\alpha }.
\end{equation*}
\end{itemize}

The significance of Theorem \ref{the1.1} is partially revealed by its
corollaries. The theorem shows us the clue to choose convergent LPA: First,
we need $X_{n}$ ($n\in
\mathbb{N}
$) to guarantee the kernel approximability, which restricts the class of
choices for $X_{n}$. Then, we choose specific $X_{n}$ ($n\in
\mathbb{N}
$) such that their offset angles $\theta _{n}$ are as small as possible,
because the small angles can not only guarantee the convergence but also
give us faster rate of convergence.

The kernel approximability of LPA is often easy to be satisfied because it
is a much weaker condition compared to the condition (\ref{eqn1.3}) (For
instance, in the case of $T$ being an injection, it is always true). Theorem %
\ref{the1.0} presents several equivalent conditions of kernel
approximability, which are:
\begin{eqnarray*}
\text{kernel approximability} &\Longleftrightarrow &\,\underset{n\rightarrow
\infty }{\mathrm{s\hspace{0.15em}}\text{-}\lim }P_{\mathcal{N}\left(
T_{n}\right) }=P_{\mathcal{N}\left( T\right) } \\
&\Longleftrightarrow &\,\underset{n\rightarrow \infty }{\mathrm{s\hspace{%
0.15em}}\text{-}\lim }P_{\mathcal{G}\left( T_{n}^{\dag }\right) }=P_{%
\mathcal{G}\left( T^{\dag }\right) } \\
&\Longleftrightarrow &\text{bounded-weak convergence.}
\end{eqnarray*}%
These equivalent characterizations of kernel approximability could help us
better understand it. From these equivalent characterizations, we can see
that the kernel approximability is a weaker version of convergence (that is
also the reason why we name one of its equivalent characterizations as
bounded-weak convergence).

The interesting points about offset angle and kernel approximability are the
geometrical revelations they show. To be more specific, the offset angle $%
\theta _{n}$ is just the largest canonical angle between $T^{\ast }T(X_{n})$
and $T^{\dagger }T(X_{n})$, which are subspaces of $\mathcal{N}\left(
T\right) ^{\bot }$; the kernel approximability is defined by the sequence $%
\left\{ \mathcal{N}\left( T\right) \cap X_{n}\right\} $, in which each
element $\mathcal{N}\left( T\right) \cap X_{n}$ is a subspace of $\mathcal{N}%
\left( T\right) $. Thus the angle between $T^{\ast }T(X_{n})$ and $%
T^{\dagger }T(X_{n})$ in ${\mathcal{N}(T)}^{\perp }$, and the kernel
approximability determined by the spaces $\mathcal{N}(T)\cap X_{n}$ in $%
\mathcal{N}(T)$, together geometrically depict the convergence of
least-squares projection method.

\section*{Acknowledgement}
The author Nailin Du's research was partially supported by LIESMARS of Wuhan
University (904110354) and China National Natural Science Foundations
(61179039 and 61273215).

\end{document}